\documentclass[10pt,a4paper]{amsart}
\normalfont
\usepackage{amssymb}
\usepackage{color}
\usepackage{amsmath} \usepackage{amscd}

\font\tenmsb=msbm10 \font\sevenmsb=msbm7 \font\fivemsb=msbm5
\newfam\msbfam
\textfont\msbfam=\tenmsb \scriptfont\msbfam=\sevenmsb
\scriptscriptfont\msbfam=\fivemsb

\font\teneufm=eufm10 \font\seveneufm=eufm7 \font\fiveeufm=eufm5
\newfam\eufmfam
\textfont\eufmfam=\teneufm \scriptfont\eufmfam=\seveneufm
\scriptscriptfont\eufmfam=\fiveeufm

\def\co#1{%{\sl#1}
}

\renewcommand{\epsilon}{\varepsilon}
\renewcommand{\setminus}{\smallsetminus}
\renewcommand{\emptyset}{\varnothing}

\newtheorem{theorem}{Theorem}[section]
\newtheorem{proposition}[theorem]{Proposition}
\newtheorem{corollary}[theorem]{Corollary}
\newtheorem{lemma}[theorem]{Lemma}

\newtheorem{question}[theorem]{Question}

\theoremstyle{definition}
\newtheorem{example}[theorem]{Example}

\newtheorem{definition}[theorem]{Definition}

\theoremstyle{remark}
\newtheorem{remark}[theorem]{Remark}

\newcommand{\pd}{\operatorname{pd}}
\newcommand{\cd}{\operatorname{cd}}
\newcommand{\vcd}{\operatorname{vcd}}

\newcommand{\normal}{\lhd}

\newcommand{\FF}{\mathcal{F}}

\newcommand{\Q}{\mathbb Q}
\newcommand{\Z}{\mathbb Z}

\newcommand{\R}{\mathbb R}
\newcommand{\C}{\mathbb C}
\newcommand{\F}{\mathbb F}
\newcommand{\T}{\mathbb T}

\newcommand{\GL}{\operatorname{GL}}

\newcommand{\fpinfty}{{\FP}_{\infty}}

\newcommand{\FP}{\operatorname{FP}}

\newcommand{\cohom}[3]{H^{{\raise1pt\hbox{$\scriptstyle#1$}}}(#2\>\!,#3)}
\newcommand{\tatecohom}[3]%
  {\widehat H^{{\raise1pt\hbox{$\scriptstyle#1$}}}(#2\>\!,#3)}

\newcommand{\Cohom}[3]%
  {H^{{\raise1pt\hbox{$\scriptstyle#1$}}}\big(#2\>\!,#3\big)}
\newcommand{\Tatecohom}[3]%
  {\widehat H^{{\raise1pt\hbox{$\scriptstyle#1$}}}\big(#2\>\!,#3\big)}

\newcommand{\homol}[3]{H_{{\lower1pt\hbox{$\scriptstyle#1$}}}(#2\>\!,#3)}
\newcommand{\homolog}[2]{H_{{\lower1pt\hbox{$\scriptstyle#1$}}}(#2)}

\newcommand{\eg}{\underline{\operatorname{E}}G}
\newcommand{\egamma}{\underline{\operatorname{E}}\Gamma}

\newcommand{\blah}{-}

\title[Subgroup posets]{Subgroup posets, Bredon cohomology and equivariant Euler characteristics.}
\author{Conchita Mart\'{\i}nez-P\'erez}

\address{C.~Mart\'inez-P\'erez. IUMA. Departamento de Matem\'aticas, Universidad de Zaragoza,
50009 Zaragoza, Spain} \email{conmar@unizar.es}
\thanks{Partially supported by BFM2010-19938-C03-03, Gobierno de Arag\'on and European Union's ERDF funds}
\date\today


\begin{document}

\begin{abstract} For discrete groups $\Gamma$ with a bound on the order of their finite subgroups, we construct Bredon projective resolutions of the trivial module in terms of projective covers of the chain complex associated to the poset of finite subgroups. We use this to give new results on dimensions of $\egamma$ and to reprove that for virtually solvable groups, $\underline{\cd}\Gamma=\vcd\Gamma$. We also deduce a formula to compute the equivariant Euler class of $\egamma$ for $\Gamma$ virtually solvable of type $\FP_\infty$ and use it to compute orbifold Euler characteristics.
\end{abstract}

\maketitle

\section{Introduction}

\noindent Already thirty years ago, Brown realized of the
relevance of the poset $\FF$ of the finite subgroups of an arbitrary (discrete) group $\Gamma$
to problems related with Euler characteristics. This poset has also been used by several authors such as  Connolly-Kozniewsky (\cite{connkoz}), Kropholler-Mislin (\cite{krophollermislin}) and L\"uck (\cite{lueck}), to construct models of $\egamma$ and to prove results on the minimal dimension of $\egamma$ for groups having a bound on the orders of their finite subgroups. Recall that a $\Gamma$-$CW$-complex $X$ is a model for $\egamma$ if $X^H\simeq\ast$ if $H$ is finite and empty otherwise, this kind of spaces have received a lot of attention in the last ten years.
The key property of the poset $\FF$ that allow these constructions can be stated non rigorously as follows: for $\FF_1:=\FF\setminus\{1\}$, the pair of spaces $(C|\FF_1|,|\FF_1|)$ where $|\FF_1|$ is the geometric realization of the poset $\FF_1$ and $C$ is the cone construction, is (non-equivariantly) homotopy equivalent to the biggest subcomplex of $\egamma$ where the group $\Gamma$ acts freely. This fact was first discovered by Connolly and Kozniewsky (\cite{connkoz}) and is closely related to the property that $|\FF_1|^H\simeq\ast$ whenever $1\neq H$ is finite which is also crucial in Brown's work.

 In this paper, we also exploit those properties of $\FF$. First, we provide a new algebraic proof of the existence of the homotopy equivalence above which works in a more general setting (for example, for arbitrary families $\mathcal{H}$ of subgroups and arbitrary coefficient fields) and
then use it to construct a Bredon projective resolution of the trivial
module in terms of projective resolutions of the chain complexes associated to certain posets of subgroups which we denote $\mathcal{H}_H$ (see Section \ref{first}). Recall that this kind of resolutions are, in the case when $\mathcal{H}=\FF$, the algebraic counterpart of the classifying spaces $\egamma$. Using
hypercohomology, this allows us to
determine $\cd_\mathcal{H}\Gamma$, the minimal length of such a projective resolution as follows:

\medskip

\noindent{\bf Theorem A. }{\it Let $\Gamma$ be a group with a bound on the $\mathcal{H}$-lengths of the subgroups in $H\in\mathcal{H}.$
%and with $\underline\cd\Gamma<\infty$.
Then
$${\cd}_{\mathcal{H}}\Gamma=\max_{H\in\mathcal{H}}\pd_{WH}\Sigma\widetilde{\mathcal{H}}_{H\bullet}.$$}

\medskip

\noindent Here, $\widetilde{\mathcal{H}}_{H\bullet}$ is the augmented chain complex associated to the poset $\mathcal{H}_H$,
%$${\mathcal{H}}_K=\{K\in\mathcal{H}\mid H<K\},$$
$WH:=N_G(H)/H,$ and $\pd_{WH}$ denotes the projective dimension of a $WH$-complex (see Section \ref{first}). This result remains valid if we work with coefficients in any commutative ring $R$ with a unity. Unfortunately, this invariant is not easy to compute. So we restrict ourselves to the family of finite subgroups to prove the following result (here and in the rest of the paper we omit the subindex $R$ when we talk about projective or cohomological dimensions).
\medskip

\noindent{\bf Theorem B. }{\it  Let $\Gamma$ be a group with a bound on the orders of the finite subgroups and with
$\underline{\cd} \Gamma<\infty$. Assume also that there is an order reversing integer valued function $l:\FF\to\Z$ such that
for each $H\leq \Gamma$ finite $\pd_{WK}B(WH)\leq l(H)$ and
\begin{itemize}
\item[i)] Either $l(K)<l(H)$ for any $H<K$

\item[ii)] or $|\FF_H|\simeq\ast$.
\end{itemize}
Then $\underline{\cd} \Gamma\leq l(1)$.}

\medskip

In this result, $\underline{\cd}\Gamma:=\cd_\FF \Gamma$ and $B(WH)$ is the $WH$-module of bounded functions first defined in \cite{krophollertalelli} (see Section \ref{first}). For virtually torsion free groups $\Gamma$, $\pd_\Gamma B(\Gamma)=\vcd\Gamma$. Using Theorem B we reprove in Section \ref{solvable} the known fact that for virtually-(torsion free solvable) groups, $\underline{\cd}\Gamma=\vcd\Gamma$ (\cite{martineznucinkis}). This is known to be false for arbitrary groups because of examples constructed in \cite{LN}. We are able to do that since for this kind of groups there is a nice invariant which can play the role of the function $l(H)$ in the previous Theorem. To be more precise, we let $l(H)$ be either the Hirsch rank of the centralizer $hC_\Gamma(H)$ or $hC_\Gamma(H)+1$ according to wether our ambient group $\Gamma$ is of type $\FP_\infty$ or not. As a by-product, we also get a characterization of when $|\FF_H|\simeq\ast$ for virtually-(torsion free solvable) groups of type $\FP_\infty$  (see Theorem \ref{contractible}).

The knowledge of when $|\FF_H|\simeq\ast$ is specially useful to compute Euler classes of elementary amenable groups of type $\FP_\infty$. Given an arbitrary group $\Gamma$ one may define the Grothendieck group $A(\Gamma)$ of proper cocompact
$\Gamma$-sets and associate to each proper cocompact $\Gamma$-$CW$-complex $X$ an element
$\chi^\Gamma(X)\in A(\Gamma)$ (see \cite{luckL2}, Definition 6.84) which is invariant under equivariant homotopy equivalence. If $\Gamma$ is elementary amenable of type $\FP_\infty$, $\Gamma$ is virtually-(torsion free solvable) and there is a cocompact $\egamma$ (\cite{kmn}). We get the following formula for $\chi^\Gamma(\egamma)$:

\medskip

\noindent{\bf Theorem C. }{\it Let $\Gamma$ be elementary amenable of type $\FP_\infty$. Then
$$\chi^\Gamma(\underline{\operatorname{E}}\Gamma)=\sum_{F\in\Omega/\Gamma}([\Gamma/F]+\sum_{H<F\atop C_G(H)>1}{1\over|F:H|}e(\widetilde{\mathcal B}_{H\bullet}^F)[\Gamma/H]).$$}

Here, $\Omega$ is a set of maximal finite subgroups of $\Gamma$,
 $\widetilde{\mathcal B}_{H\bullet}^F$ is the augmented chain complex of certain poset of subgroups inside the finite group $F$ and $e(\widetilde{\mathcal B}_{H\bullet}^F)$ denotes its ordinary Euler number (see Section \ref{euler}). In some cases, all the relevant parts of this formula can be computed using elementary character theory and first integral cohomology groups. An interesting case when this formula applies is to compute Euler classes and orbifold Euler characteristics of toroidal orbifolds, we do it in several examples at the end of the paper.

\section{Algebraic version of a Theorem by Connolly-Kozniewski}\label{first}

\noindent In this Section, we are going to prove Theorems A and B.
%To prove Theorem A, we will construct a Bredon projective resolution of the trivial module for
Let $\mathcal{H}$ be an arbitrary  class of subgroups of $\Gamma$, here following \cite{symonds} we mean by class just a set of subgroups which is closed under conjugation.
Recall that a Bredon contramodule is a contravariant functor from the category of transitive $\Gamma$-sets with stabilizers in $\mathcal{H}$ and $\Gamma$-maps as morphisms to the category of $R$-modules ($R$ is a commutative ring with a unity). The category of Bredon contramodules is abelian and can be shown to have enough projectives so one can define Ext functors in the usual way. The trivial module is the constant functor with value $R$ and is denoted by $\underline R$.
 Given $H\in\mathcal{H}$, let $P^\Gamma_H$ be the functor which takes any other $T\in\mathcal{H}$ to the free $R$-module generated by the maps from $\Gamma/T$ to $\Gamma/H$, i.e. $P^\Gamma_H(T):=R<(\Gamma/H)^T>$. 
This $P^\Gamma_H$ is called the free contramodule based at $H$. Note that if $X$ is a $\Gamma$-$CW$-complex , we may associate to each $H\in\mathcal{H}$ the chain complex of $X^H$ and we get a chain complex of Bredon contramodules for $\mathcal{H}$, which is called the Bredon chain complex of $X$ respect to $\mathcal H$. If the cell stabilizers of $X$ belong to $\mathcal{H}$, this chain complex consists of free modules.
As in ordinary cohomology, $\cd_{\mathcal{H}}\Gamma$ is the minimal length of a Bredon projective resolution of the trivial module $\underline{R}$. For more precise definitions in Bredon cohomology the reader is referred to \cite{LN}, \cite{martineznucinkis} and \cite{conch}.

From now on, by chain complex we mean vanishing below zero chain complex.
Let $C_\bullet$ be a chain complex of right $\Lambda$-modules for a ring
$\Lambda$. A projective resolution of $C_\bullet$ (see for example \cite{BenI} 2.7)
is a chain complex $P_\bullet$ of projective right $\Lambda$-modules together with a map of complexes $P_\bullet\to C_\bullet$ which is a quasi-isomorphism or weak-equivalence
(i.e., yields an isomorphism on homology). Existence and uniqueness up to homotopy of
projective resolutions is easily proven as in the case of projective resolutions of
single modules (which corresponds to chain complexes concentrated in degree 0). Using
projective resolutions of chain complexes one may define Ext and Tor for chain complexes.

%In particular if $\Lambda=R\Gamma$ is a group ring we may define groups cohomology with
%coefficients the chain complex $C_\bullet$ as
%$$H^n(\Gamma,C_\bullet):=\text{Ext}^n_{R\Gamma}(R_\bullet,C_\bullet)$$
%where $R_\bullet$ denotes the chain complex with $R$ in degree $n=0$ and $0$ elsewhere.
We denote by
$\text{pd}_{R\Gamma}C_\bullet$
the projective dimension of the chain complex $C_\bullet$, that is, the shortest length
of a projective resolution of $C_\bullet$. As in the case of modules,
$d=\text{pd}_{R\Gamma}C_\bullet$ can be also defined as the only integer such that
$\text{Ext}_\Lambda^d(C_\bullet,D_\bullet)\neq0$
for some chain complex $D_\bullet$ and
$\text{Ext}_\Lambda^m(C_\bullet,-)=0$
for any $m>d$. We let the projective dimension of the zero complex (and of any exact complex) be $-\infty$. 
From now on we will assume that $\Lambda=R\Gamma$ is a group ring.
Let $C_\bullet$ be a chain complex of $R\Gamma$-modules such that there is a morphism $C_0\to
R$. Then
we denote by $\widetilde C_\bullet:=C_\bullet\to R$ the augmented chain complex. We also consider the suspension $\Sigma\widetilde C_\bullet$ given by
$\Sigma\widetilde C_n=\widetilde C_{n-1}.$
%Therefore if $X$ is a CW-complex and $CX$ denotes its topological cone, $$\Sigma\widetilde{\mathcal C}(X)_\bullet=\mathcal C(CX,X)_%\bullet.$$

\begin{lemma}\label{tec1} Let $U_\bullet$, $V_\bullet$ be quasi-isomorphic $\Gamma$-chain complexes. Then
$$\pd_\Gamma U_\bullet=\pd_\Gamma V_\bullet$$
\end{lemma}
\proof Let $P_\bullet\twoheadrightarrow U_\bullet$, $Q_\bullet\twoheadrightarrow V_\bullet$ be projective resolutions of both chain complexes, then the Comparison Theorem (see for example \cite{BenI} 2.2.7 and 2.4.2) implies that there is a quasi-isomorphism between $P_\bullet$ and $Q_\bullet$. As both are complexes of projectives, they are $\Gamma$-homotopy equivalent so both are projective covers of either of the two chains.
\qed

\begin{lemma}\label{tec2} Let $U_\bullet$, $V_\bullet$ be $\Gamma$-chain complexes which are exact everywhere except of degree 0 where they have homology $R$. Assume that $U_\bullet$ consists of projectives. Then there is a $\Gamma$-quasi-isomorphism
$$U_\bullet\to V_\bullet.$$
\end{lemma}
\proof There are obvious quasi-isomorphisms $V_\bullet\buildrel\alpha\over\to R_\bullet$ and $U_\bullet\to R_\bullet$
where $R_\bullet$ is the complex with $R$ concentrated in degree 0. Let $Q_\bullet$ be a projective resolution of $V_\bullet$, then there is a quasi-isomorphism
$Q_\bullet\to V_\bullet$ which can be composed with $\alpha$ to give
$$Q_\bullet\to R_\bullet.$$

The Comparison Theorem implies that the identity on $R_\bullet$ can be lifted to a quasi-isomorphism $Q_\bullet\to U_\bullet$. As these are chain complexes of projectives, it is in fact a homotopy equivalence so there is also a homotopy equivalence $U_\bullet\to Q_\bullet$. Therefore there is a quasi-isomorphism  $U_\bullet\to Q_\bullet\to V_\bullet$.\qed

\iffalse{
\begin{lemma}\label{connkoz}(\cite{connkoz}) Let $H\leq \Gamma$ be a finite subgroup and $X$ be a $\Gamma$-$CW$-complex such that
for any $K\in\FF_H$, $X^K\simeq\ast$. Then
there is a $WH$-map
$$\sigma_H(X)\to|\FF_H|$$
which is a (non-equivariant) homotopy equivalence.
\end{lemma}}\fi
Let $\mathcal P$ be a $\Gamma$-poset.
We denote by ${\mathcal P}_\bullet$  the chain complex of $\Gamma$-modules
associated to the geometrical realization  $|\mathcal P|$.

For the proof of the next result, it is useful to work in the category of local coefficient systems for the poset $\mathcal{H}$, a category which is closely related to that of Bredon modules. A local coefficient system with $R$-coefficients  for $\mathcal{H}$ is a contravariant functor from the category having as objects the transitive $\Gamma$-sets with stabilizers in $\mathcal{H}$ but with morphisms given only by inclusions $T\leq S$ with $S,T\in\mathcal{H}$ to the category of $R$-modules.  Obviously, any Bredon contramodule is also a local coefficient system. The category of local coefficient systems is abelian and contains enough projectives too. Its free objects are
$$S_K(T):=\Bigg\{\begin{aligned}R\text{ if }T\leq K\\
0\text{ otherwise}\\
\end{aligned}$$
Free Bredon contramodules are also free as local coefficient systems, in fact
\begin{equation}\label{freebredon}
P_K^\Gamma=\bigoplus_{x\in\Gamma/K}S_{K^{x^{-1}}}.
\end{equation}

  For  any $H\leq\Gamma$ we put
 $$\mathcal{H}_H=\{H<K:K\in\mathcal{H}\}.$$
 %or just $\mathcal{H}_H$ if the group $\Gamma$ is clear by the context.
 %If $X$ is a $G$-space let
 %$$\sigma_H(X):=\{x\in X:H<G_x\}.$$
%$$\Lambda^1(WH)=\{H<K\leq N_\Gamma(H)\}$$

\begin{lemma}\label{homot} Let $H\leq \Gamma$ be a  subgroup with $\mathcal{H}_H\neq\emptyset$ and
$C_\bullet$ be a chain complex of free Bredon contramodules based at
$\mathcal{H}_H$ such that
$C_\bullet(K)\twoheadrightarrow R$ is exact for each $K\in\mathcal{H}_H$.
Then there is a $WH$-quasi-isomorphism $\Sigma\widetilde C_\bullet(H)\to\Sigma\widetilde{\mathcal{H}}_{H\bullet}$ and
$$\pd_{WH}\Sigma\widetilde C_\bullet(H)=\pd_{WH}\Sigma\widetilde{\mathcal{H}}_{H\bullet}.$$
\end{lemma}
\proof Observe that $C_\bullet$ is a free resolution of $\underline{R}$ in the category of Bredon contramodules for the class $\mathcal{H}_H$.
For any $T\in\mathcal{H}_H$, let ${\mathcal{H}}_{T\leq}:=\{K\in\mathcal{H}\mid T\leq K\}$. Then, the map
$${\mathcal{H}}_{(-)\leq_\bullet}:T\mapsto{\mathcal{H}}_{T\leq\bullet}$$
yields a chain complex of Bredon contramodules. Obviously, $|{\mathcal{H}}_{T\leq}|$
is contractible so this chain complex is exact everywhere except of degree 0 where its homology is the trivial object.
Now, by the analogous statement of \ref{tec2} in the category of Bredon contramodules there is a quasi-isomorphism
\begin{equation}\label{quasiiso}C_\bullet(\blah)\to {\mathcal{H}}_{(-)\leq\bullet}.\end{equation}
Obviously, these chain complexes also belong to the category of local coefficient systems. The advantage here is that ${\mathcal{H}}_{(-)\leq\bullet}$ is build from projectives, in fact
$$({\mathcal{H}}_{(-)\leq })_t=\bigoplus\{S_K(-)\mid K\in{\mathcal{H}}_H\text{ such that }K=K_0<\ldots< K_t\text{ for some }K_i\in\mathcal{H}\}.$$

 So we have a quasi-isomorphism between chain complexes of projectives which must be therefore an homotopy equivalence of local coefficient systems. Next, we take the colimit of the diagrams obtained by evaluating both chain complexes at each $T\in\mathcal{H}_H$ and with arrows given by inclusions.
Note that for any $K\in\mathcal{H}_H$,
$$\text{colim}_{T\in\mathcal{H}_H}S_K(T)=R=S_K(H)$$
From this and (\ref{freebredon}) follows that also $$\text{colim}_{T\in\mathcal{H}_H}P^\Gamma_K(T)=P^\Gamma_K(H)$$
As our chain complexes are built from modules of either of these two kinds we deduce that
$$\text{colim}_{T\in\mathcal{H}_H}C_\bullet(T)=C_\bullet(H),$$
$$\text{colim}_{T\in\mathcal{H}_H}{\mathcal{H}}_{T\leq\bullet}={\mathcal{H}}_{H\leq \bullet}={\mathcal{H}}_{H\bullet}.$$

Moreover, (\ref{quasiiso}) induces a map between both diagrams, so we get a map
$$C_\bullet(H)\to {\mathcal{H}}_{H\leq\bullet}=\mathcal{H}_{H\bullet}$$
which is an $R$-homotopy equivalence. We claim that it is also a $WH$-map (but not a $WH$-homotopy equivalence).
 To see it, note that from the obvious conjugation action of $WH$ on $\mathcal{H}_H$, we get an action of $WH$ in both diagrams. Observe that any $gH\in N_\Gamma(H)/H=WH$ yields a $\Gamma$-map $\Gamma/K\to \Gamma/K^{g}$ and since (\ref{quasiiso}) is a Bredon morphism, the map that it induces commutes with the $WH$-action. So when we take the direct limit we do get an $WH$-map.
 Now, the fact that it is an $R$-homotopy equivalence implies that it induces isomorphisms between the homology groups, in other words, it is a $WH$-quasi-isomorphism which extends to a quasi-isomorphism $\Sigma\widetilde C_\bullet(H)\to\Sigma\widetilde{\mathcal{H}}_{H\bullet}$. For the last assertion use \ref{tec1}.
\qed

\begin{remark} The previous result remains true if we take as $C_\bullet$ a resolution build from projective Bredon contramodules for ${\mathcal{H}}$. \end{remark}

\iffalse{Let $X$ be a model for $\eg \Gamma$. Then
with respect to the family $\FF_H$ both $C_\bullet$ and
$\mathcal{C}(\sigma_H(X)^{-})$ are resolutions by projective
Bredon modules of the constant functor $\underline{R}$. Therefore
there is a homotopy equivalence of chains of Bredon modules for
$\FF_H:$
$$C_\bullet\simeq\mathcal{C}(\sigma_H(X)^{-}).$$
We may induce to the whole family $\FF$ as in \cite{symonds}, note
that $\text{lim}_{\FF_H}^{\FF}C_\bullet=C_\bullet$ by \cite{symonds}
6.2 so the homotopy above extends a homotopy equivalence of Bredon
modules for $\FF$. Evaluation in $H$ yields a $\Gamma$-homotopy equivalence
$$C_\bullet(H)\simeq\mathcal{C}(\sigma_H(X)).$$
Now by \ref{connkoz} there is a $\Gamma$-map
$$\mathcal{C}(\sigma_H(X))_\bullet\to(\FF_H)_\bullet$$ which is a
quasi-isomorphism. Composing it with the previous homotopy equivalence we get a quasi-isomorphism
$$C_\bullet(H)\to(\FF_H)_\bullet$$
so it suffices to use Lemma \ref{tec1}.}\fi

Consider a subgroup $H\leq \Gamma$. We can associate to every $WH$-module $U$ a Bredon contramodule  for $\mathcal{H}$ denoted $\text{ind}_\tau U$ via the following formula (see \cite{lueckbook})
$$\text{ind}_\tau U(\blah)=P^\Gamma_H(\blah)\otimes_{WH}U.$$
The notation is due to the fact that $\text{ind}_\tau$ is an induction functor. It is the left adjoint to the restriction functor, which consists just on evaluation at $H$. This implies that $\text{ind}_\tau$ takes projectives to projectives, a fact that also follows by observing that $\text{ind}_\tau RWH=P^\Gamma_H$.
Note that for any $WH$-module $U$,
$$\text{ind}_\tau U(H)=P^\Gamma_H(H)\otimes_{WH}U=RWH\otimes_{WH}U=U.$$
We are going to use this in Theorem \ref{resolution} to construct explicitly free Bredon resolutions for groups having a bound on the $\mathcal{H}$-lengths of their subgroups in $\mathcal{H}$, where the $\mathcal{H}$-length of a group $H$ is the biggest $s$ such that there is a chain
$$1\leq H_0<H_1<\ldots<H_s\leq H$$
of subgroups with all $H_i\in\mathcal{H}$.
%{\bf Important: check if I have to add at some point the hypothesis that the class $\mathcal{H}$ is subgroup closed.}

\begin{theorem}\label{resolution} Let $\Gamma$ be a group with a bound on the $\mathcal{H}$-lengths of the subgroups in $\mathcal{H}$ and take for any $H\in\mathcal{H}$, a $WH$-free resolution $P_{H\bullet}$ of the chain complex $\Sigma\widetilde{\mathcal{H}}_{H\bullet}$.
%and with $\underline\cd\Gamma<\infty$.
Then there is an $\mathcal{H}$-Bredon projective resolution of $\underline R$ which regardless of the connecting maps is
$$\bigoplus_{H\in\mathcal{H}/\Gamma}\text{ind}_\tau P_{H\bullet}=\bigoplus_{H\in\mathcal{H}/\Gamma}P^\Gamma_{H}(-)\otimes_{WH}P_{H\bullet}$$
($\mathcal H/\Gamma$ is a set of representatives of $\Gamma$-orbits in $\mathcal{H}$).
\end{theorem}
\proof
%For each $l$, we construct inductively a chain complex consisting of projective Bredon  contramodules based at subgroups of length at %least $l$ which behaves as a resolution of the trivial module when evaluating at those subgroups and has length
%$$\max_{l\leq l(H)}\pd_{WH}\Sigma\widetilde{\mathcal{H}}_H.$$
We proceed inductively, constructing on each step a Bredon projective resolution of $\underline R$ for
$\{H\in\mathcal{H}\mid H\text{ has $\mathcal{H}$-length}\geq l\}$.

To do that, assume that the resolution $C_\bullet$ is constructed for $l+1$ (we allow the possibility of $C_\bullet$ being the zero complex to include here the case when $H$ is of maximal length, i.e., the first step of the inductive process). Let $H\in\mathcal{H}$ of length $l$ and note that  $C_\bullet$ is exact when
evaluating at each $K\in\mathcal{H}_H$. Essentially, the process below consists in adding free contramodules based at subgroups of length exactly $l$.
% to get a complex which is also a resolution of the trivial module when we evaluate at those subgroups, as those free modules vanish %at subgroups of length $>l$ we do not loose the previous properties.
%Let $d=\pd_{WH}\Sigma\widetilde C(H)_\bullet=\pd_{WH}\Sigma\widetilde{\mathcal{H}}_H$ (by Lemma \ref{homot})
 %and choose a projective cover
 %$P_{H\bullet}$ of length $d$  of $\Sigma\widetilde C_\bullet(H)$ so that
 By Lemma \ref{homot} there is a $WH$-quasi-isomorphism $\Sigma\widetilde C_{\bullet}(H)\to\Sigma\widetilde{\mathcal{H}}_{H\bullet}$. From the Comparison Theorem follows that the quasi-iso $P_{H\bullet}\to\Sigma\widetilde{\mathcal{H}}_{H\bullet}$ can be lifted to a new quasi-iso
$$P_{H\bullet}\to\Sigma\widetilde C_{\bullet}(H).$$
The adjoint isomorphism for chain complexes
$$\text{Hom}_{WH}(P_{H\bullet},\Sigma\widetilde C_\bullet(H))\cong\text{Hom}_\mathcal{H}(\text{ind}_\tau P_{H\bullet},\Sigma\widetilde C_\bullet)$$
yields a chain map
$$\text{ind}_\tau P_{H\bullet}\to\Sigma\widetilde C_\bullet.$$
The same can be done with any representative of the (possibly infinitely many) conjugacy classes of subgroups of length $l$. Summing them up we get a new chain map
$$b:\bigoplus_{\{H\in\mathcal{H}\mid l(H)=l\}/\Gamma}\text{ind}_\tau P_{H\bullet}\to\Sigma\widetilde C_\bullet.$$

Now let $Q_\bullet$ be the chain complex of Bredon contramodules
such that its augmented complex is  $\widetilde
Q_\bullet=\text{cone}b_{\bullet+1},$ (see \cite{Weibel}, Section 1.5) i.e.,
$$Q_n=\bigoplus_{\{H\in\mathcal{H}\mid l(H)=l\}/\Gamma}\text{ind}_\tau P_{Hn}\oplus C_{n}.$$
For any $K\in\mathcal{H}_H$ with $l(K)>l$, $Q_\bullet(K)=C_\bullet(K)$ and
$Q_\bullet(H)\twoheadrightarrow R$
is also exact, so $Q_\bullet$ is the desired resolution.
\qed

\begin{theorem}\label{dimension} If there is a bound on the $\mathcal{H}$-lengths of subgroups in $\mathcal{H}$,then
$${\cd}_{\mathcal{H}}\Gamma=\max_{H\in\mathcal{H}}\pd_{WH}\Sigma\widetilde{\mathcal{H}}_{H\bullet}.$$
\end{theorem}
\proof By Theorem \ref{resolution}, $\cd_\mathcal{H}\Gamma\leq\max_{H\in\mathcal{H}}\pd_{WH}\Sigma\widetilde{\mathcal{H}}_{H\bullet}$ so we only have to prove that for any $H\in\mathcal{H}$, $\pd_{WH}\Sigma\widetilde{\mathcal{H}}_{H\bullet}\leq{\cd}_{\mathcal{H}}\Gamma$ and we may assume $\cd_{\mathcal{H}}\Gamma<\infty$.  Basically by the same argument as in the ordinary case one can prove that there is a free resolution $P_\bullet$ of $\underline R$ of length $\cd_{\mathcal{H}}\Gamma$. Consider the subcomplex $C_\bullet$ of $P_\bullet$ formed by all those free modules based at subgroups $K$ with $H<_\Gamma K$. Evaluating at $H$ we get the $WH$-short exact sequence
 $$\Sigma\widetilde C_\bullet(H)\rightarrowtail\Sigma\widetilde P_\bullet(H)\twoheadrightarrow \Sigma F_\bullet$$
 where $\Sigma F_\bullet$ is a chain complex consisting of free $WH$-modules and has length bounded by ${\cd}_{\mathcal{H}}\Gamma+1$. The fact that
$\Sigma\widetilde P_\bullet(H)$ is exact implies that its Ext functors vanish so the long exact sequence yields
$$\pd_{WH}\Sigma\widetilde C_\bullet(H)=\pd_{WH}\Sigma F_\bullet-1=\text{length}\Sigma F_\bullet-1\leq\cd_{\mathcal{H}}\Gamma$$
and the result follows by Lemma \ref{homot}.
\qed

\begin{remark} Theorems \ref{resolution} and \ref{dimension} can be seen as an algebraic version of
\cite{connkoz} Theorem III but for arbitrary families of subgroups and coefficients rings; moreover, we do not assume the group is virtually torsion free. See also \cite{symonds} Proposition 8.6, where a similar result is proven in the case of a finite group and the family of its $p$-subgroups.
\end{remark}

%\begin{remark} Lemma \ref{homot} implies that for any $CW$-complex $X$ which is a model for $E_\mathcal{H} G$, then
%$$\pd_{WH}\Sigma\widetilde{\mathcal{H}}_H\leq\text{dim}\{\sigma\mid\sigma\text{ cell in $X$ such that }\text{Stab}_G\sigma=H\}.$$
%\end{remark}

\iffalse{
\begin{remark} Note that for chain complexes obtained from spaces, the notion of projective dimension can also be stated in terms of Borel cohomology. Namely, let $X$ be a $\Gamma$-CW-complex and $C_\bullet$ the associated chain complex of permutation $\Gamma$-modules. Then for any $\Gamma$-module $M$
$$\text{H}^n_\Gamma(E\Gamma\times (CX,X),M)=\text{Ext}^n_\Gamma(\Sigma\widetilde C_\bullet,M)=\text{H}^n(\text{Hom}_\Gamma(P_\bullet,M))$$
where $P_\bullet$ is a projective resolution of $\Sigma\widetilde C_\bullet$. Therefore
$$\text{pd}\Sigma\widetilde C_\bullet=\text{max}\{n\mid\text{H}^n_\Gamma(E\Gamma\times (CX,X),M)\neq 0\text{ for some module }M\}.$$
\end{remark}}\fi

\begin{remark}\label{lueck1} An inductive process as in Theorem \ref{resolution} but with $\mathcal{H}_{1(WH)}=\{H<K\leq N_\Gamma(H)\}$ instead of $\mathcal{H}_H$ does not work because of the following: Assume we have constructed $C_\bullet$, a Bredon free resolution for ${\mathcal{H}_H}_\bullet$ and let $\sigma_{WH}C_\bullet(-)$ be the subcomplex consisting of those free summands of $C_\bullet$ based at subgroups $K$ with $H<_\Gamma N_G(H)\cap K$. It is a consequence of Lemma \ref{homot} that $\pd_{WH}\Sigma\widetilde{\mathcal{H}}_{1(WH)\bullet}=\pd_{WH}\Sigma\sigma_{WH}\widetilde C_\bullet(H).$ Moreover, $C_\bullet(H)$ and $\sigma_{WH}C_\bullet(H)$ only differ on free $WH$-modules. We may compare their projective dimensions after augmenting and suspending using the short exact sequence
$$\Sigma\sigma_{WH}\widetilde C_\bullet(H)\rightarrowtail\Sigma\widetilde C_\bullet(H)\twoheadrightarrow \Sigma F_\bullet$$
and we get
\begin{equation}\label{pdsigma}\begin{aligned}
\pd_{WH}\Sigma\widetilde C_\bullet(H)\leq\max\{\pd_{WH}\Sigma\sigma_{WH}\widetilde C_\bullet(H),\pd_{WH}\Sigma F_\bullet\}\leq\\
\max\{\pd_{WH}\Sigma\sigma_{WH}\widetilde C_\bullet(H),\text{length}C_\bullet+1\}.
\end{aligned}\end{equation}
This essentially means that it might be the case that to make $C_\bullet$ exact at $H$, we have to add free contramodules based at $H$ at degrees bigger than $\pd_{WH}\Sigma\widetilde{{\mathcal{H}}}_{1(WH)\bullet}$. For a group with a subgroup $H$ such that $\pd_{WH}\Sigma\widetilde{{\mathcal{H}}}_{1(WH)\bullet}<\pd_{WH}\Sigma\widetilde{\mathcal{H}}_{H\bullet}=\pd_{WH}\Sigma\widetilde C_\bullet(H)$ see Example \ref{learynucinkis}.
\end{remark}

%This is also related to \cite{lueck} 1.6. But I seriously suspect that Luck did a mistake, so  don't trust that result and probably it is better %don't mention it.

%\begin{remark} Despite of Remark \ref{lueck1}, it could be true that in general, in general
%$$\text{max}_{H\in\mathcal{H}}\pd_{WH}{\mathcal{H}}_H=\text{max}_{H\in\mathcal{H}}\pd_{WH}{\mathcal{H}}_1(WH).$$
%\end{remark}

 From now on, unless otherwise stated we will concentrate in the case when the family is $\FF=\mathcal{F}in$, i.e. the family of finite subgroups of $\Gamma$ and we denote $\underline{\cd}\Gamma=\cd_\FF\Gamma$.
Note that Theorem \ref{dimension} implies that if $\Gamma$ is assumed to have bounded orders of finite subgroups,  $\underline{\cd}\Gamma<\infty$ if and only if $\max_{H\in\mathcal{H}}\pd_{WH}\Sigma\widetilde{\FF}_{H\bullet}<\infty$. Unfortunately, as remarked before, this last invariant is very difficult to compute. In recent years, several authors have tried to characterize algebraically groups $\Gamma$ with $\underline{\cd}\Gamma<\infty$. A group invariant of importance for this purpose is  $\pd_\Gamma B(\Gamma)$ where $B(\Gamma)$ is the module of bounded functions on $\Gamma$ first defined by Kropholler and Talelli in \cite{krophollertalelli} as
$$B(\Gamma):=\{f:\Gamma\to R\mid f(\Gamma)\text{ takes finitely many distinct values}\}.$$
The main properties that make this module interesting in this context are that it contains a copy of the trivial module $R$ and that it is free when restricted to every finite subgroup of $\Gamma$.
Using the first of these properties it is easy to prove (see \cite{conch} Lemma 3.4) that for modules $M$ of finite projective dimension $\pd_\Gamma M=\pd_\Gamma M\otimes B(\Gamma)$ (here and below by $\otimes$ we mean $\otimes_R$).
Moreover, whenever we have subgroups $H\leq K$, then
$\pd_{WK}B(WK)\leq\pd_{WH}B(WH)$ (\cite{conch} Proposition 3.7). Recall also that for groups $\Gamma$ of finite virtual cohomological dimension, $\pd_\Gamma B(\Gamma)=\vcd\Gamma$.

 \begin{lemma}\label{induction} Assume that $\Gamma$ has a bound on the orders of its finite subgroups. Then if $\underline{\cd}\Gamma<\infty,$
$$\pd_{WH}\Sigma\widetilde{\FF}_{H\bullet}\leq\max\{\pd_{WH}B(WH),\max_{H<K}\pd_{WK}\Sigma\widetilde\FF_{K\bullet}+1\}.$$
\end{lemma}
\proof Note first that our assumption on $\underline{\cd}\Gamma$ and Theorem \ref{resolution} imply that for any finite $H$, $\pd_{WH}\Sigma\widetilde{\FF}_{H\bullet}<\infty$.
By the analogous of \cite{conch} Lemma 3.4 for chain complexes we have $\pd_{WH}\Sigma\widetilde{\FF}_{H\bullet}=\pd_{WH}\Sigma\widetilde{\FF}_{H\bullet}\otimes B(WH).$
Now, consider the short exact sequence of $WH$-chain complexes
$$R_\bullet\rightarrowtail\Sigma\widetilde{\FF}_{H\bullet}\twoheadrightarrow\Sigma{\FF}_{H\bullet}$$
and tensor it with the module $B(WH)$
\begin{equation}\label{seq2}
B(WH)_\bullet\rightarrowtail\Sigma\widetilde{\FF}_{H\bullet}\otimes B(WH)\twoheadrightarrow\Sigma{\FF}_{H\bullet}\otimes B(WH).\end{equation}

Using the long exact sequence of Ext functors we get
$$\pd_{WH}\Sigma\widetilde{\FF}_{H\bullet}\otimes B(WH)\leq\text{max}\{\pd_{WH}B(WH),\pd_{WH}{\FF}_{H\bullet}\otimes B(WH)+1\}.$$

So to finish the proof we only have to show that $$\pd_{WH}{\FF}_{H\bullet}\otimes B(WH)\leq\max_{H<K}\pd_{WK}\Sigma\widetilde\FF_{K\bullet}.$$ Observe first that the simplices of $|\FF_H|$ are of the form $K=K_0<\ldots<K_t$ with $H<K$. We may use this to filter $\FF_{H\bullet}$ in terms of the size of the bottom subgroup (i.e., the $K$) of each simplex. Each term of this filtration is a direct sum of the $WH$-chain complexes of simplices having a group $WH$-conjugated to a fixed $K$ at the bottom, and this chain complex is precisely
$(\Sigma\widetilde{\FF}_{K\bullet})\uparrow_{W_KH}^{WH}$
where we denote $W_KH=(N_\Gamma(K)\cap N_\Gamma(H))/H$.  Therefore we deduce

$$\pd_{WH}{\FF}_{H\bullet}\otimes B(WH)\leq\max_{K\in\FF_H/WH}\pd_{WH}(\Sigma\widetilde{\FF}_{K\bullet})\uparrow_{W_KH}^{WH}\otimes B(WH)$$

Obviously $
\pd_{WH}(\Sigma\widetilde\FF_{K\bullet})\uparrow_{W_KH}^{WH}\otimes B(WH)=\pd_{W_KH}\Sigma\widetilde\FF_{K\bullet}\otimes B(WH)$.
Now, $B(WH)$ is, as $W_KH$-module, a direct summand of $B(W_KH)$ so
$$\pd_{W_KH}\Sigma\widetilde\FF_{K\bullet}\otimes B(WH)\leq \pd_{W_KH}\Sigma\widetilde\FF_{K\bullet}\otimes B(W_KH).$$
Moreover, $C_\Gamma(K)\leq N_\Gamma(H)\cap N_\Gamma(K)\leq N_\Gamma(K)$ is of finite index and therefore
$\pd_{W_KH}\Sigma\widetilde\FF_{K\bullet}=\pd_{WK}\Sigma\widetilde\FF_{K\bullet}<\infty$ and
$$\pd_{W_KH}\Sigma\widetilde\FF_{K\bullet}\otimes B(W_KH)=\pd_{W_KH}\Sigma\widetilde\FF_{K\bullet}=\pd_{WK}\Sigma\widetilde\FF_{K\bullet}.$$

\qed

\begin{example}\label{learynucinkis} Let $\Gamma=K\ltimes H_L$ be the group constructed in \cite{LN} Section 9 Example 4, where $K=A_5$ is the alternating group of degree 5, and $H_L$ the Bestvina-Brady group associated to a certain flag complex $L$ of dimension 2. The non-trivial subgroups of $K$ are, up to conjugation $C\cong C_2$, a 2-Sylow $P\cong C_2\times C_2$, $N_K(P)\cong A_4$, a 3-Sylow $Q\cong C_3$, $N_K(Q)\cong S_3$, a 5-Sylow $R\cong C_5$ and $N_K(R)\cong D_{10}$. Using the explicit description of $L$ in \cite{LN} one easily sees that all the fixed points subcomplexes $L^T$ when $1<T<K$ are contractible and have dimension less or equal than 1 (more precisely, $L^P$, $L^{N_K(P)}$, $L^R$, $L^{N_K(R)}$, $L^{N_K(Q)}$ are points and $L^C$, $L^Q$ have dimension 1), whereas $L^K=\emptyset$. Then Theorem 3 of \cite{LN} implies that for any such $T$ there is only one conjugacy class of complements of $H_L$ in $TH_L$ and that $WT\cong (N_K(T)/T)\ltimes H_{L^T}$. Moreover, for any $S\leq WT$ finite, $(L^T)^S$ is contractible (observe that $K$ is simple, so it is not contained in any $N_\Gamma(T)$). Thus by \cite{LN} Theorem 5
$$\underline{\cd}WT=\vcd WT=\cd H_{L^T}=\dim L^T<2=\dim L$$
and using Theorem \ref{dimension} we get $\pd_{WT}\Sigma\widetilde{\FF}_{1(WT)\bullet}\leq\underline{\cd}WT<2$ (we are using the same notation as in Remark \ref{lueck1} but for $\FF$).
On the other hand, for those $A_5\cong K_1\leq\Gamma$, we have $WK_1=1$ so $\pd_{WK_1}\Sigma\widetilde\FF_{K_1\bullet}=0$. All this means that for any $1\neq T\leq\Gamma$ finite, $\pd_{WT}\Sigma\widetilde{\FF}_{1(WT)\bullet}<2$. Assume now that $\pd_{WT}\Sigma\widetilde\FF_{T\bullet}\leq\pd_{WT}\Sigma\widetilde{\FF}_{1(WT)\bullet}$ for all such $T$. Then
Lemma \ref{induction} would imply
$$\pd_\Gamma\Sigma\widetilde\FF_{1\bullet}\leq\max\{\vcd\Gamma,2\}=2$$
so by Theorem \ref{dimension} we deduce $\underline{\cd}\Gamma\leq 2$ which contradicts \cite{LN} Theorem 6.

\iffalse{ In this case $\pd_\Gamma B(\Gamma)=\vcd\Gamma=\cd H_L$. Moreover, it is easy to check that in this example for any $1\neq T\leq Q$, $\cd C_{H_L}(T)=\cd H_{L^T}<\cd H_L$ and $$.  This means that if for any such $T$ we had $\pd_{WT}(\Sigma\widetilde\FF_T)_\bullet=\pd_{WT}(\Sigma\widetilde\FF_1(WT))_\bullet\leq\underline{\vcd}WT=\vcd WT$ then {\bf Only careful here is I am counting all the conjugacy classes or I have to do something. \cite{LN} Theorem 3 implies that if $L^T\neq\emptyset$, then there is only one conjugacy class of subgroups that map to $T$ or some of its $Q$-conjugated subgroups.}
So there is some $T$ such that $$\pd_{WT}(\Sigma\widetilde\FF_T)_\bullet>\pd_{WT}(\Sigma\widetilde\FF_1(WT))_\bullet.$$}\fi
\end{example}

\begin{remark}\label{resfinsubgroups} Let $H\leq \Gamma$ be a finite subgroup. Assume $\Gamma$ has bounded lengths of finite subgroups. Then one can check that for all $1\neq K/H\leq WH$ finite
$$\pd_{K/H}\Sigma\widetilde{\FF}_{H\bullet}<\infty.$$
By \cite{corkro} Theorem C, if $\Gamma\in H\FF$ (the class of hierarchically decomposable groups, see \cite{corkro}) and $\pd_\Gamma B(\Gamma)<\infty$, then every module which has finite projective dimension when restricted to every finite subgroup has finite projective dimension with respect to the whole group $\Gamma$. This implies that $\pd_{WH}\Sigma\widetilde{\FF}_{H\bullet}<\infty$.
\end{remark}

\begin{theorem}\label{condition} Let $\Gamma$ be a group with a bound on the orders of its finite subgroups and with
$\underline{\cd}\Gamma<\infty$. Assume also that there is an order reversing integer valued function $l:\FF\to\Z$ such that
for each $H\leq \Gamma$ finite $\pd_{WH}B(WH)\leq l(H)$ and
\begin{itemize}
\item[i)] Either $l(K)<l(H)$ for any $H<K$

\item[ii)] or $|\FF_H|\simeq\ast$.
\end{itemize}
Then $\underline{\cd}\Gamma\leq l(1).$\end{theorem}
\proof Fix a finite subgroup $H$. By an inductive argument we may assume that for any $H<K$,
$\pd_{WK}\Sigma\widetilde{\FF}_{K\bullet}\leq l(K)$ (note that in particular this holds if $K$ is maximal).
We claim that $\pd_{WH}\Sigma\widetilde{\FF}_{H\bullet}\leq l(H)$. This is obvious in case ii), so we assume i).
 We have then
 $$\max_{H<K}\pd_{WK}\Sigma\widetilde{\FF}_{K\bullet}+1\leq \max_{H<K}l(K)+1\leq l(H)$$
 Using Lemma \ref{induction} and the hypothesis we get the claim so the result follows by Theorem \ref{dimension}.
\qed

\section{Virtually-(torsion free solvable) groups.}\label{solvable}

\noindent We begin this section by showing that virtually-(torsion free solvable) groups $\Gamma$ satisfy the hypothesis of Theorem \ref{condition} with a function $l$ closely related to the Hirsch rank $h\Gamma$. The first results will be useful for inductive arguments.

\begin{lemma}\label{boundedsection} Any bounded section of an abelian group of finite Pr\"ufer rank is finite.
\end{lemma}
\begin{proof} As any section of an abelian group of finite Pr\"ufer rank has also finite Pr\"ufer rank, we only have to prove that if $A$ is bounded, abelian and of finite Pr\"ufer rank then it is finite. 
   As $A$ is bounded, we may assume that all its elements have order a power of a fixed prime $p$. By considering the Kernel and Image of the group homomorphism
 $\beta:A\to A$ given by $\beta(a)=a^p$, one easily sees that we may assume that all the elements of $A$ have order $p$. But then $A$ must be  a finite dimensional vector space over $\F_p$.
\end{proof}

\begin{lemma}\label{inductionsolv} Let $\Gamma$ be virtually-(torsion free solvable) of finite Hirsch rank, $A\normal\Gamma$ abelian and $\bar \Gamma =\Gamma /A$. For any $L\leq\Gamma$ finite
$$hC_\Gamma(L)=hC_A(L)+hC_{\bar\Gamma}(L).$$
Moreover, whenever $H<L$, $|C_\Gamma (H):C_\Gamma (L)|$ is finite if and only if $|C_{\bar \Gamma }(H):C_{\bar \Gamma }(L)|$ and $|C_A(H):C_A(L)|$ are both finite.
\end{lemma}
\begin{proof} Note first that $A$ is virtually-(torsion free abelian) of finite Hirsch rank and therefore it has finite Pr\"ufer rank. Using Lemma \ref{boundedsection} and with the same proof as in \cite{martineznucinkis} Lemma 2.2 we deduce that $H^1(H,A)$ is finite and by the proof of \cite{martineznucinkis} Lemma 3.10 there is an injection
$$C_{\bar \Gamma }(H)/AC_\Gamma (H)\rightarrowtail H^1(H,A).$$
From this both assertions follow easily.
\end{proof}

\begin{theorem}\label{solv} Let $\Gamma$ be virtually-(torsion free solvable) and
$H<L\leq\Gamma$ finite subgroups. Then the following are equivalent:
\begin{enumerate}

\item $h C_\Gamma(H)=h C_\Gamma(L)$,

\item $|C_\Gamma(H):C_\Gamma(L)|<\infty$,

\end{enumerate}
\end{theorem}
\begin{proof} %Let $G\normal\Gamma$ torsion free of finite index. As $G$ is torsion free and of finite Hirsch rank, it must be nilpotent-by-free abelian of finite rank-by-finite (see for example \cite{baerheineken} Proposition 5.5). In fact, going down finite index subgroups if necessary we may assume that there is a $N\normal G$ nilpotent with $G/N$ free abelian of finite rank and also that $N$ is $L$-invariant.
Obviously, (ii) implies (i). So we assume (i) and let $Ht_1,\ldots,Ht_s$ be the left classes of $H$ in $L$, thus $s=|L:H|$. By a result of Robinson (see \cite{LennoxRobinson} Theorem 5.2.5) there is a finite index characteristic subgroup $G$ of $\Gamma$ with a characteristic series of finite length whose factors are torsion-free abelian.

Assume first that the group $A:=G$ is torsion free abelian. Let $e:=\sum_i t_i$ and $B:=C_A(H)$. The group $B$ might not be $L$-invariant but $Be\subseteq B_1:=C_A(L)\subseteq B$.
Then
$$Bs=Be+B(s-e)$$
and $B(s-e)\subseteq B_2=:\text{Ann}_B(e)$. Moreover the fact that $B$ is torsion free implies $B_1\cap B_2=0$. Now from the hypothesis $hB_1=hB$ we deduce $hB_2=0$ and again the fact that $B$ is torsion free forces to $B_2=0$. So we have $bs\in B_1$ for any $b\in B=C_A(H)$. But this means that for each $l\in L$, $b(l-1)s=0$ and as $B$ is torsion free we see that $b\in B_1=C_A(L)$.
For the general case use induction and Lemma \ref{inductionsolv}.

% so we assume (iii). Recall that for a virtually - (torsion free solvable) group $S$, $\vcd S=hS\text{ or }hS+1$ according to whether $S$ is $\FP_\infty$ or not. Observe also that then (iii) implies that we may assume that $C_\Gamma(H)$ is $\FP_\infty$ and it will suffice to prove that then also $C_\Gamma(L)$ is $\FP_\infty$.
%In fact, it suffices to prove that $(C_G(H)\cap N)C_G(L)$ is of type $\FP_\infty$ and then deduce the result using \cite{martineznucinkis} Lemma 3.12.

%Arguing by induction on $|L:H|$, we may assume that either $H\normal L$ or $L=<H^t\mid t\in L>$. In the first case, note that the finite %group $L/H$ acts on $C_G(H)$ so the result follows by \cite{martineznucinkis} Theorem 3.13.
%So we assume $L=<H^t\mid t\in L>$ and consider the following map $$
%\begin{aligned}
%C_G(H)/C_N(H)\to C_G(L)/C_N(L)\\
%C_N(H)b\mapsto C_N(L)b^{t_1}\ldots b^{t_s}.\\
%\end{aligned}$$
%Note that the fact that $N\normal G$ is $L$-invariant implies that this is well defined.$C_G(L)/C_N(L)\to C_G(L)C_N(H)/C_N(H)$

\end{proof}

\begin{lemma}\label{keypoint} Let $\Gamma$ be a group with $G\normal\Gamma$ torsion free. Assume that for some $g\in G$ and $L\leq H<\Gamma$ both finite, $L^g\leq H$. Then $L=L^g$. As a consequence,  $L\leq H^x$ for $x\in G$ implies $x\in C_G(L)$.
\end{lemma}
\begin{proof} The fact that the natural projection $\pi:H\to HG/G$ is a group isomorphism implies that whenever $L_1\leq H$ is such that $LG=L_1G$, then $L=L_1$. In particular, observe that $L^gG=LG$.
\end{proof}

\begin{proposition}\label{contractible} Let $\Gamma$ be virtually-(torsion free solvable) and $H\leq\Gamma$ finite. If there is some $H<T$ finite with $|C_\Gamma(H):C_\Gamma(T)|<\infty$, then $|\FF_H|\simeq\ast$.
\end{proposition}
\begin{proof} We claim that for any $L\in\FF_H$,  $<T,L>$ is finite. This will mean that the poset $\FF_H$ is directed thus contractible.

Let $G$ be as in the proof of Theorem \ref{solv} and assume first that $A=G$ is torsion free abelian. Let $S=<TA,LA>$, obviously $F:=S/A$ is finite and we have a group extension
$$1\rightarrow A\rightarrow S\rightarrow F\rightarrow 1$$
which corresponds to a class $\alpha\in H^2(F,A)$. As $\alpha$ has finite order, say $l$, the map $A\to A$ consisting on multiplying by $l$ can be extended to an embedding $S\to F\ltimes A$, which can be composed with the embedding $F\ltimes A\to F\ltimes V=:S_1$ with $V=A\otimes\Q$. So we may assume $H,T,L\leq S_1$ and $C_{V}(T)=C_{V}(H)$
(note that $C_{V}(K)=C_A(K)\otimes\Q$ for $K=L,H$). If we prove that $L,T$ generate a finite subgroup of $S_1$, it follows that they also generate a finite subgroup of $S$. This means that we may assume that our original group is $\Gamma=F\ltimes A$ with $A=\Q^n$ and $\Gamma/A=<TA,LA>/A$. As the cohomology of a finite group with coefficients in a $\Q$-module vanishes, for any group between $A$ and $\Gamma$ the corresponding extension splits and there is only one conjugacy class of subgroups that gives the splitting. In particular, $F$ contains $A$-conjugates of $L$ and $T$. By changing to some $A$-conjugated subgroup of $F$ if necessary, we may assume that $L\leq F$ and $T^a\leq F$ for some $a\in A$. But then $H,H^a\leq F$ so by Lemma \ref{keypoint}, $a\in C_A(H)=C_A(T)$ thus $L,T\leq F$.

For the general case, argue by induction on the Hirsch length of $G$.
Take $A\normal G$ abelian and $\Gamma$-invariant and let $\bar G=G/A$. By Lemma \ref{inductionsolv} we have the same hypothesis for the group $\Gamma/A$ with respect to the finite subgroups $\bar H=HA/A$, $\bar T=TA/A$ and $\bar L=LA/A$. So we may assume by induction that the group $S/A=<\bar T,\bar L>=<T,L>A/A$ is finite. Therefore $L,H,T\leq S$ which is virtually abelian and again by Lemma \ref{inductionsolv} we have the same hypothesis so by induction (or by the virtually abelian case) the claim follows.

\end{proof}

\noindent The results proven so far imply

\begin{corollary}\label{dimension2} Let $\Gamma$ be virtually-(torsion free solvable).
Then $\Gamma$ satisfies the hypothesis of Theorem \ref{condition} with $l(H)=hWH$ if $\Gamma$ is of type $\FP_\infty$ and $l(H)=hWH+1$ in other case.
In particular
$$\underline{\cd}G=\vcd G.$$
\end{corollary}
\begin{proof} In both cases, note that $l(K)=l(H)$ for $H<K\in\FF$ implies $hC_\Gamma(H)=hWH=hWK=hC_\Gamma(K)$. By Theorem \ref{solv} then $|C_\Gamma(H):C_\Gamma(K)|<\infty$ and Proposition \ref{contractible} implies $|\FF_H|\simeq\ast$. We only have to check that $\pd_{WH}B(WH)\leq l(H)$. Observe first that in this case $\pd_{WH}B(WH)=\vcd WH\leq hWH+1$ as our groups are virtually (torsion free-solvable). If moreover $\Gamma$ is of type $\FP_\infty$ then \cite{martineznucinkis} Theorem 3.13 implies that also the Weyl groups $WH$ are of type $\FP_\infty$ and therefore $\vcd WH=hWH.$
\end{proof}

\noindent As a by-product we can prove

\begin{theorem}\label{contractible2}  The converse to Proposition \ref{contractible} is also true except of possibly the case when $WH$ is of type $\FP_\infty$ but $\Gamma$ is not. \end{theorem}
\begin{proof}
Let $G\leq\Gamma$ be a finite index torsion free subgroup of $\Gamma$. Take $H\leq\Gamma$ finite with $\FF_H\simeq\ast$. Then the Ext functors of the complex $\Sigma\widetilde{\FF}_{H\bullet}$ vanish so the long exact sequence associated to (\ref{seq2}) in the proof of Lemma \ref{induction} implies that
$$\vcd WH=\pd_{WH}B(WH)=\pd_{WH}B(WH)\otimes\Sigma{\FF}_{H\bullet}-1
=\pd_{WH}B(WH)\otimes{\FF}_{H\bullet}$$
and by that proof
$$\pd_{WH}B(WH)\otimes{\FF_H}_\bullet\leq\max_{H<K}\pd_{WK}\Sigma\widetilde\FF_{K\bullet}.$$
But for $l(H)$ as in Corollary \ref{dimension2}, we have proven in Theorem \ref{condition} that for any $H<K$ finite,
$$\pd_{WK}\Sigma\widetilde\FF_{K\bullet}\leq l(K)\leq l(H).$$
This means that there is some $H<K$ finite with
$\vcd WH\leq l(K)\leq l(H).$ If $\Gamma$ is $\FP_\infty$ or if $WH$ is not $\FP_\infty$, $l(H)=\vcd WH$ so we deduce $l(H)=l(K)$.
\end{proof}

We end this section with an example of an explicit computation of a Bredon projective resolution using the inductive procedure of Theorem \ref{resolution}.

\begin{example}\label{explicitBredon} Let $\Gamma=F\ltimes A$ with $A=<a,b>\cong\Z^2$, $F=S_3$ and the action of $F$ on $A$ given by the following matrix representation of $F$ (with $x^2=y^3=1$):
$$x=\begin{pmatrix}
0&1\\
1&0\\
\end{pmatrix},\quad
y=\begin{pmatrix}
0&-1\\
1&-1\\
\end{pmatrix}.$$
An easy computation of first cohomology groups shows that $\{1,<x>,<y>,$ $F,<ya>\}$ is a complete set of representatives of the $\Gamma$-conjugacy classes of finite subgroups.

The inductive procedure yields the following free resolution of $\Z$ by Bredon contramodules (but note that we do not describe the connecting maps):

$$P_1^\Gamma\to P_{<x>}^\Gamma\bigoplus_2 P_1^\Gamma\to P_1^\Gamma\oplus P_F^\Gamma\oplus P_{<ya>}^\Gamma.$$

%{\bf This process also allows to determine the singularizations.}
\end{example}

\iffalse{
\begin{remark}\label{characters}
Question \ref{preg} can be reformulated with the language of character theory. Let $\Gamma$ be virtually-(torsion free abelian). Then there is a short exact sequence
$$1\to A\to\Gamma\to F$$
with $A$ torsion free abelian and $F$ finite. The group $F$ acts on $A$ and on $V=A\otimes_\Z\Q$, thus $V$ is an $F\Q$-module. Let $\chi$ denote the (rational) associated character, and for any $H_1\leq\Gamma$ finite let $H=AH_1/A\leq F$. We have
$hC_A(H_1)=\text{dim}_\Q V^{H}=(\chi_{H},1_{H})$ where the right hand side denotes the Frobenius product. Therefore if there is some $H<T$ such that
$(\chi_{T},1_{T})=(\chi_{H},1_H)$, then $\FF_{H_1}\simeq\ast.$ Is the converse true?
\end{remark}}\fi

\section{Equivariant Euler characteristics for $\underline{\operatorname E}\Gamma$ for solvable groups.}\label{euler}

\noindent Following \cite{luckL2}, Section 6 we let $A(\Gamma)$ be the Grothendieck group of finitely generated proper (i.e., with finite stabilizers) $\Gamma$-sets. We denote the class in $A(\Gamma)$ of a transitive set having $S$ as a point stabilizer by $[\Gamma/S]$. Given a $\Gamma$-$CW$-complex $X$, we say that the $\Gamma$-action on $X$ is admissible if whenever a simplex is fixed setwise, it is fixed pointwise.

\begin{definition} (\cite{luckL2}, Definition 6.84) Let $X$ be a proper cocompact $\Gamma$-$CW$-complex with an admissible $\Gamma$-action. The equivariant Euler class of $X$ is
$$\chi^\Gamma(X):=\sum_{\sigma\in X/\Gamma}(-1)^{\text{dim}\sigma}[\Gamma/\Gamma_\sigma]\in A(\Gamma)$$
where $\Gamma_\sigma$ denotes the stabilizer in $\Gamma$ of the cell $\sigma$.
\end{definition}

In this section we compute a formula for $\chi^\Gamma(X)$ when $\Gamma$ is an elementary amenable group of type $\FP_\infty$ and $X$ a cocompact model for $\underline{E}\Gamma$ (which exists by \cite{kmn}). 

\begin{definition} (\cite{Brown2}) Let $\Gamma$ be a group and $P_\bullet$ a finite chain complex of free $\Gamma$-modules. The equivariant Euler characteristic of $P_\bullet$ is
$$e^{\Gamma}(P_\bullet)=\sum_t(-1)^t\text{rk}_\Z(P_t\otimes_{\Z\Gamma}\Z).$$
\end{definition}

This definition can be extended to finite chains of projectives using Hattori-Stallings ranks, we do not need to do that here as all the complexes we consider are build of free modules. It can also be extended to arbitrary chain complexes of $\Gamma$-modules $C_\bullet$ which admit a finite free (or just projective) resolution $P_\bullet$ via
$$e^{\Gamma}(C_\bullet):=e^{\Gamma}(P_\bullet).$$
In particular, we can do this for single $\Gamma$-modules (seen as chain complexes concentrated in degree 0). Then by \cite{Brown2} we have:
\begin{equation}\label{sum}
e^{\Gamma}(C_\bullet)=\sum_t(-1)^te^{\Gamma}(C_t)
\end{equation}

Let $P_\bullet$ be a finite chain complex of free Bredon contramodules.
For each degree $t$, let $C_{Ht}$ be the cokernel of the map
$$\text{colim}_{L\in\FF_H}P_t(L)\rightarrowtail P_t(H).$$
Then we have a chain complex of $WH$-modules $C_{H\bullet}$.
Note that if we let $C_\bullet$ be as in Theorem \ref{dimension} the subcomplex of $P_\bullet$ given by those free contramodules based at $\FF_H$, then the same argument of Lemma \ref{homot} implies that $\text{colim}_{L\in\FF_H}P_\bullet(L)=C_\bullet(H)$. In other words, $C_{H\bullet}=P_\bullet(H)/C_\bullet(H)$ thus $C_{H\bullet}$ consists of free  $WH$-modules.

\begin{definition}\label{eulerchcomplex}  With $P_\bullet$, $C_{H\bullet}$ as before, the equivariant Euler class of $P_\bullet$ is
$$\chi^\Gamma(P_\bullet):=\sum_{H\in\FF/\Gamma}e^{WH}(C_{H\bullet})[\Gamma/H]$$
where $\FF/\Gamma$ denotes the set of orbits in $\FF$ under the $\Gamma$-conjugation action.
\end{definition}

\begin{remark}\label{welldefined} Assume that the chain complexes of Bredon contramodules $P_\bullet$ and $P'_\bullet$ are Bredon homotopy equivalent. This means that for any finite $H\leq\Gamma$, the $WH$-complexes $P_\bullet(H)$ and $P'_\bullet(H)$ are also homotopy equivalent, moreover all these homotopies are compatible with the morphisms induced by $\Gamma$-maps $\Gamma/H\to\Gamma/L$. Therefore there is a commutative diagram of $WH$-chain complexes
\begin{equation}\label{diagram}
\begin{aligned}
&\text{colim}_{L\in\FF_H}P_\bullet(L)&\rightarrowtail &P_\bullet(H)\\
&\quad\downarrow&&\downarrow\\
&\text{colim}_{L\in\FF_H}P'_\bullet(L)&\rightarrowtail &P'_\bullet(H)\\
\end{aligned}
\end{equation}
where the vertical arrows come from the given homotopy equivalence $P_\bullet\to P'_\bullet$ and are also homotopy equivalences. %Moreover, note that the first part of the proof of Lemma \ref{homot} implies that there are short exact sequences
%$$\text{colim}_{L\in\FF_H}P_\bullet(L)\rightarrowtail P_\bullet(H)\twoheadrightarrow C_{H\bullet}(H);$$
%$$\text{colim}_{L\in\FF_H}P'_\bullet(L)\rightarrowtail P'_\bullet(H)\twoheadrightarrow (P')^H_\bullet(H),$$
%or in other words, that $P_\bullet(H)/\text{colim}_{L\in\FF_H}P_\bullet(L)\cong C_{H\bullet}H)$ and the same for $P'_\bullet.$
The fact that all the relevant homotopies commute with the morphisms induced by $\Gamma$-maps implies that the homotopy equivalences of (\ref{diagram}) extend to a $WH$-homotopy equivalence between the cokernels, i.e., between $C_{H\bullet}$ and $C'_{H\bullet}$. Thus the equivariant Euler class is invariant under Bredon homotopy equivalences. \end{remark}

 The next Lemma is basically \cite{luckL2} Lemma 6.85 but we prove it here for completeness.

\begin{lemma}\label{eulercomplex} Let $X$ be a proper cocompact $\Gamma$-$CW$-complex and $P_\bullet$ the associated Bredon chain complex. Then $$\chi^\Gamma(P_\bullet)=\chi^\Gamma(X).$$
\end{lemma}
\begin{proof} Recall that $P_\bullet(H)$ is the chain complex of $X^H$. It is a consequence of the observation before Definition \ref{eulerchcomplex} that (using the same notation as there), 
$$|\{\sigma\in X^H/WH;\dim\sigma=t,H=\Gamma_\sigma\}|=\text{rk}_{\Z}C_{Ht}\otimes_{\Z\Gamma}\Z.$$
 Therefore 
$$\begin{aligned}
\chi^\Gamma(X):=\sum_{\sigma\in X/\Gamma}(-1)^{\text{dim}\sigma}[\Gamma/\Gamma_\sigma]
=\sum_{H\in\FF/\Gamma}\sum_{\sigma\in X/\Gamma\atop H=_\Gamma\Gamma_\sigma}(-1)^{\text{dim}\sigma}[\Gamma/H]\\
=\sum_{H\in\FF/\Gamma}\sum_{\sigma\in X^H/WH\atop H=\Gamma_\sigma}(-1)^{\text{dim}\sigma}[\Gamma/H]=\chi^{\Gamma}(P_\bullet).\\
\end{aligned}$$\end{proof}

Remark \ref{welldefined} and Lemma \ref{eulercomplex} have the consequence that if the $\Gamma$ $CW$-complexes $X$ and $Y$ are equivariantly homotopy equivalent, then $\chi^\Gamma(X)=\chi^\Gamma(Y).$ In particular, we may talk of $\chi^\Gamma(\underline{\text{E}}\Gamma)$, as this does not depend on the explicit model of $\underline{\text{E}}\Gamma$ that we choose. Note also that if we denote as in Theorem \ref{resolution} by ${P_H}_\bullet$ a projective resolution of $\Sigma\widetilde{\FF}_{H\bullet}$ which is finite and free, then

\begin{equation}\label{firstobs}\begin{aligned}\chi^\Gamma(\underline{E}\Gamma)
=\sum_{H\in\FF/\Gamma}e^{WH}({P_H}_\bullet)[\Gamma/H].\\
\end{aligned}\end{equation}

From now on we fix a group $\Gamma$ which we assume to be elementary amenable of type $\FP_\infty$. Recall that such $\Gamma$ is virtually-(torsion free solvable), moreover as we have seen in the proof of \ref{solv}, it has a finite index normal subgroup $G$ with a chain with torsion free abelian factors.
\iffalse{Let $H\leq\Gamma$ finite, we claim that the complex $\Sigma\widetilde{\FF}_{H\bullet}$ has a finite free resolution.  Let $\sigma_HX=\cup_{K\in\FF_H}X^K$ and consider the chain complex of Bredon contramodules for $\FF_H$ associated to $\sigma_HX$ given, for $K\in\FF_H$, by $$\sigma_HC_\bullet(K)=\text{chain complex of the space }\sigma_HX)^K.$$ Consider the associated chain complexes of Bredon contramodules for $\FF_H$ $C_\bullet$ and $\sigma_HC_\bullet$, so that for any $K\in\FF_H$, $C_\bullet(K)$ and $\sigma_H(K)$ are the chain complexes of the spaces $X^K$ and $(\sigma_HX)^K$ respectively. As $C_\bullet(H)$ is contractible, we deduce that $$
$$\Sigma\widetilde{\sigma_HC}_\bullet(H)\rightarrowtail\Sigma\widetilde C_\bullet(H)\twoheadrightarrow$$
 let $C_\bullet$,  be the associated chain complex of Bredon contramodules for
consider, for any $H\in\FF$,
$(X^H,\sigma_H(X))$ where  (observe that $WH$ might act freely in some cells of $\sigma_HX$). The chain complex of $(X^H,\sigma_H(X))$ seen as chain complex of $WH$-modules is free and finite, in fact it is a free finite resolution of $\Sigma\widetilde{\FF}_{H\bullet}$ which we denote by ${P_H}_\bullet$.}\fi

Now, for any finite $H\leq\Gamma,$ $C_G(H)$ is isomorphic to a finite index subgroup of $WH$, its index being $|N_\Gamma(H):HC_G(H)|$ thus

\begin{equation}\label{also}\begin{aligned}
e^{WH}({P_H}_\bullet)={|H|\over |N_\Gamma(H):C_G(H)|}e^{C_G(H)}({P_H}_\bullet).\\\end{aligned}\end{equation}

These coefficients are easy to compute at least in some cases. Observe first that as ${P_H}_\bullet$ is a free $WH$-cover of $\Sigma\widetilde{\FF}_{H\bullet}$, it is also a $C_G(H)$-free cover thus $e^{C_G(H)}({P_H}_\bullet)=e^{C_G(H)}(\Sigma\widetilde{\FF}_{H\bullet})$.
Consider the following set
$$\Omega:=\{F\leq\Gamma\mid\text{ maximal finite, }C_G(F)=1\}.$$

\begin{lemma}\label{maxfin}
Let $H\in\FF$ with $C_G(H)=1$. 
\begin{itemize}
\item[i)] If $H\in\Omega$, then $H=N_\Gamma(H)$ and $e^{C_G(H)}({P_H}_\bullet)=e^{WH}({P_H}_\bullet)=1$.

\item[ii)] In other case, $e^{C_G(H)}({P_H}_\bullet)=e^{WH}({P_H}_\bullet)=0.$
\end{itemize}
\end{lemma}
\begin{proof} In case i), as $H$ is maximal finite, $\Sigma\widetilde{\FF}_{H\bullet}=R_\bullet$ concentrated at degree 0. Moreover the condition $C_G(H)=1$ implies that $N_\Gamma(H)$ is finite, so $H=N_\Gamma(H)$.
In case ii), Lemma \ref{contractible} implies $\FF_H\simeq\ast$ thus $\Sigma\widetilde{\FF}_{H\bullet}$ is exact and has the zero chain complex as projective cover.
\end{proof}

\noindent To compute $e^{C_G(H)}({P_H}_\bullet)$ in general, we proceed as follows.  First recall that as our groups are amenable, whenever $G\neq 1$, $e^G(\Z)=0$ (see \cite{Eckmann}).
 
We may decompose $e^{C_G(H)}(\Sigma\widetilde{\FF}_{H\bullet})$ using (\ref{sum}) as for $t>0$, $\Sigma\widetilde{\FF}_{Ht}$ is  a sum of permutation modules of the form $\Z_{N_\Gamma(\sigma)}\uparrow^{N_\Gamma(H))}$ where $\sigma$ runs over the $WH$-orbits of $t$-cells $H<H_1<\ldots H_t$. As $N_\Gamma(\sigma)\cap C_G(H)=\cap N_\Gamma(H_i)\cap G=C_G(H_t)$, the cell $\sigma$ lies in a free $C_G(H)$-orbit if and only if $C_G(H_t)=1$ and in other case we have
$$e^{C_G(H)}(\Z_{N_\Gamma(\sigma)}\uparrow^{N_\Gamma(H)})=\sum_x e^{C_G(H)}(\Z_{C_G(H_t)^{x^{-1}}}\uparrow^{C_G(H)})=\sum_x e^{C_G(H_t)^{x^{-1}}}(\Z)=0$$
where $x$ runs over the (finite) set $N_\Gamma(\sigma)\backslash N_\Gamma(H)/C_G(H)$.
This means that we only have to consider the summands associated to cells $\sigma$ lying on free $C_G(H)$-orbits. Also, if $C_G(H)\neq 1$ we have
$e^{C_G(H)}(\Z)=0$ thus

\begin{equation}\label{decom}
\begin{aligned}
e^{C_G(H)}(\Sigma\widetilde{\FF}_{H\bullet})&=
e^{C_G(H)}(\Sigma{\FF}_{H\bullet})=-e^{C_G(H)}({\FF}_{H\bullet})\\
&=-\sum\{(-1)^{t}\mid \sigma\in\text{free $C_G(H)$-orbits of $t$-cells in }\FF_H\}.\\
\end{aligned}
\end{equation}
In particular, if $H$ is not contained in any subgroup $F$ with $C_G(F)=1$, then $e^{C_G(H)}({P_H}_\bullet)=0=e^{WH}({P_H}_\bullet)$.

\begin{lemma}\label{unique} Let $S\leq \Gamma$ be finite with $C_G(S)=1$. Then there is a unique $F\in\Omega$ such that $S\leq F$.
\end{lemma}
\begin{proof} Choose $F$ maximal finite with $S\leq F$. Then $C_G(F)\leq C_G(S)=1$ thus $F\in\Omega$. We only have to prove the uniqueness, obviously we may assume that $S$ itself is not maximal finite. Assume that $S<F,F_1\in\Omega$. As  $|C_\Gamma(S):C_\Gamma(F)|<\infty$, the proof of Proposition \ref{contractible} implies that $<F,F_1>$ is finite so $F=F_1$.
\end{proof}

As a consequence of this Lemma, we have that if all the non-trivial finite subgroups act without non-trivial fixed points in $G$, then any finite subgroup is contained in a single maximal finite subgroup (see \cite{LangerLuck} Example 7.15 where also a particular case of the formula of Theorem \ref{mainformula} is considered).

\begin{definition} Let $F\in\Omega$ and $H<F$. We denote
$$\mathcal B_H^F:=\{H<S< F\mid 1<C_G(S)\}.$$
This is a poset. We denote by $e(\widetilde{\mathcal B}_{H\bullet}^F)$ the (ordinary) Euler characteristic of the result of augmenting its chain complex.
\end{definition}

Now we are ready to prove

\begin{theorem}\label{mainformula}
$$\chi^\Gamma(\underline{E}\Gamma)=\sum_{H\in\FF/\Gamma}(\sum_{F\in\Omega/N_\Gamma(H)\atop H\leq F}{1\over{|N_F(H):H|}}e(\widetilde{\mathcal B}_{H\bullet}^F))[\Gamma/H].$$
\end{theorem}
\begin{proof} For each $H\in\FF$, we have to compute the coefficient of $[\Gamma/H]$ in $\chi^\Gamma(\underline{E}\Gamma)$. If $C_G(H)=1$ we have already done it in Lemma \ref{maxfin}. Note that, if $H\in\Omega$, then $F=H$ and ${\mathcal B}_H^F=\emptyset$ thus $e(\widetilde{\mathcal B}_{H\bullet}^F)=1$ and if $H\not\in\Omega$, then by Lemma \ref{unique} we have ${\mathcal B}_H^F\simeq\ast$ so $e(\widetilde{\mathcal B}_{H\bullet}^F)=0$.

So from now on, we assume that $C_G(H)\neq1$.
By (\ref{firstobs}), (\ref{also}) and  (\ref{decom}) the coefficient of $[\Gamma/H]$ is determined by the free $C_G(H)$-orbits of cells in $\FF_H$. Such cells are of the form $\sigma:H<H_0<\ldots<H_t$ with $C_G(H_t)=1$ and Lemma \ref{unique} implies that $H_t$ lies in a uniquely determined $F\in\Omega$ (so every element of the $C_G(H)$-orbit of $\sigma$ corresponds to a $C_G(H)$-conjugated subgroup of $F$).  This means that
%\begin{equation}\label{coefficient}{-|H|\over |N_\Gamma(H):C_G(H)|}\sum\{(-1)^{t}\mid \sigma\in\text{free $C_G(H)$-orbits of $t$-cells}\}%\end{equation}

 \begin{equation}\label{decom2}\begin{aligned}
 \sum\{(-1)^{t}&\mid \sigma\in\text{free $C_G(H)$-orbits of $t$-cells in }\FF_H\}=\\
&\sum_{F\in\Omega\cap\FF_H/C_G(H)}\sum\{(-1)^{t}\mid H<H_0<\ldots H_t\leq F\text{ with } C_G(H_t)=1\}\\
\end{aligned}\end{equation}

We fix for a while $F\in\Omega$ with $H\leq F$ and consider the poset
$$[F,H):=\{S\mid H<S\leq F\}.$$
Then there is a short exact sequence of chain complexes:
$$0\to\mathcal B_{H\bullet}^F\to [F,H)_\bullet\to\mathcal A_\bullet\to 0$$
where $\mathcal A_\bullet$ is the $\Gamma$-chain complex of the free $\Z$-module generated by the set of chains of the form $H<H_0<\ldots H_t$ with $C_G(H_t)=1$. Considering the augmented versions of the first two chain complexes we get
$$0\to\widetilde{\mathcal B}_{H\bullet}^F\to \widetilde{[F,H)}_\bullet\to\mathcal A_\bullet\to 0.$$
As $[F,H)_\bullet$ is contractible, the ordinary Euler characteristic of $\widetilde{[F,H)}_\bullet$ vanishes, so 

\begin{equation}\label{euler}\sum\{(-1)^{t}\mid H<H_0<\ldots H_t\leq F\text{ with } C_G(H_t)=1\}
=e(\mathcal A_\bullet)=-e(\widetilde{\mathcal B}_{H\bullet}^F)
\end{equation}

At this stage, (\ref{firstobs}), (\ref{also}), (\ref{decom}), (\ref{decom2}) and (\ref{euler}) imply that the coefficient of $[\Gamma/H]$ in $\chi^\Gamma(\underline{E}\Gamma)$ is
$${|H|\over |N_\Gamma(H):C_G(H)|}\sum_{F\in\Omega/C_G(H)\atop H<F}e(\widetilde{\mathcal B}_{H\bullet}^F).
$$
To finish the proof, we only have to see what happens when we consider the elements of $\Omega$ up to $N_\Gamma(H)$-conjugacy instead of up to $C_G(H)$-conjugacy. To do that, observe that whereas all the $C_G(H)$-orbits in $\Omega$ are free (as $N_\Gamma(F)\cap C_G(H)=1$), the stabilizer of the $N_\Gamma(H)$-orbit of $F$ is $N_\Gamma(H)\cap N_\Gamma(F)=N_F(H)$  (recall that by the remark before the definition of $\Omega$, $N_\Gamma(F)=F$). This means that each $N_\Gamma(H)$-orbit can be decomposed exactly on
$|N_\Gamma(H):N_F(H)C_G(H)|$ orbits respect to $C_G(H)$. As
$${|H||N_\Gamma(H):N_F(H)C_G(H)|\over |N_\Gamma(H):C_G(H)|}={|H|\over|N_F(H)|}$$
we are done.\end{proof}

%\begin{remark} In the formula of Theorem \ref{mainformula} the sum runs over $N_\Gamma(H)$-orbits in $\Omega$. Note that if $H\leq %F,F^g$ for $g\not\in N_\Gamma(H)$, then Lemma \ref{keypoint} implies that $GF\neq G^g$. \end{remark}

The formula of Theorem \ref{mainformula} can be rewritten in terms of $\Omega$ as follows.

\begin{corollary}\label{mainformula2}
$$\chi^\Gamma(\underline{\operatorname{E}}\Gamma)=\sum_{F\in\Omega/\Gamma}([\Gamma/F]+\sum_{H<F\atop 1<C_G(H)}{1\over|F:H|}e(\widetilde{\mathcal B}_{H\bullet}^F)|H|[\Gamma/H]).$$
\end{corollary}
\begin{proof} Let $H\in\FF$ with $1<C_G(H)$ and choose some $F\in\Omega\cap\FF_H$. Let $F_1,\ldots, F_{t(H)}$ be representatives of the $N_\Gamma(H)$-orbits inside the intersection of  the $\Gamma$-orbit of $F$ with $\FF_H$. Note that we may choose $1=g_1,\ldots, g_{t(H)}\in\Gamma$ such that for $F:=F_1$, $F_i=F^{g^{-1}_i}$. Therefore, $H^{g_i}\leq F$ for any $1\leq i\leq t(H)$. And the set of subgroups of $F$ which are $\Gamma$-conjugated to $H$ is precisely
$$\bigcup_{i=1}^{t(H)}\{H^{g_ix}\mid x\in F/N_F(H^{g_i})\}.$$
Moreover $N_{F_i}(H)\cong N_{F}(H^{g_i})$ and 
$e(\widetilde{\mathcal B}_{H\bullet}^{F_i})=e(\widetilde{\mathcal B}_{H^{g_i}\bullet}^{F})$. So we get

\begin{equation}\label{cambio}\begin{aligned}
\sum_{H\in\FF/\Gamma\atop 1<C_G(H)}(\sum_{F\in\Omega\cap\FF_H/N_\Gamma(H)}{1\over{|N_F(H):H|}}e(\widetilde{\mathcal B}_{H\bullet}^F))[\Gamma/H]=\\
\sum_{H\in\FF/\Gamma\atop 1<C_G(H)}(\sum_{F\in\Omega/\Gamma\atop H<_\Gamma F}\sum_{i=1}^{t(H)}{1\over{|N_{F_i}(H):H|}}e(\widetilde{\mathcal B}_{H\bullet}^{F_i}))[\Gamma/H]=\\
\sum_{F\in\Omega/\Gamma}(\sum_{H<F\text{ up to $\Gamma$-conj.}\atop 1<C_G(H)}\sum_{i=1}^{t(H)}{1\over{|N_{F}(H^{g_i}):H^{g_i}|}}e(\widetilde{\mathcal B}_{H^{g_i}\bullet}^{F})[\Gamma/H])=\\
\sum_{F\in\Omega/\Gamma}(\sum_{H<F\atop 1<C_G(H)}{1\over{|F:N_F(H)||N_{F}(H):H|}}e(\widetilde{\mathcal B}_{H\bullet}^{F})[\Gamma/H]).
\end{aligned}\end{equation}

The result now follows using Lemma \ref{maxfin}.

\end{proof}

\iffalse{Claim:
$$\chi^\Gamma(\underline{\operatorname{E}}\Gamma)=\sum_{F\in\Omega/\Gamma}{1\over{|F|}}(\sum_{H<F}|H|(e(\widetilde{\mathcal B}_{H\bullet}^F)-e^{C_G(H)}(\Z))[\Gamma/H])+[\Gamma/F].$$

Remark: if $C_G(H)=1$, then $e^{C_G(H)}(\Z)=1$ and in other case $e^{C_G(H)}(\Z)=0$. Moreover, if $C_G(H)=1$, then $e(\widetilde{\mathcal B}_{H\bullet}^F)=1$. So in the sum we only have to consider those $H<F$ with $1<C_G(H)$ and the formula looks nicer:}\fi

Now, consider the map $f:A(\Gamma)\to\Q$ taking $f([\Gamma/H])\to{1\over|H|}$. Then $f(\chi^\Gamma(\underline{\operatorname{E}}\Gamma))$ is the ordinary equivariant Euler characteristic.
By \cite{Brown} Proposition 7.3 b'), as $\Gamma$ is amenable,
$f(\chi^\Gamma(\underline{\operatorname{E}}\Gamma))=0.$
This can also be seen using the previous formula. Note that for any $F\in\Omega$:
$$\begin{aligned}
f(\sum_{H<F\atop1<C_G(H)}{1\over|F:H|}e(\widetilde{\mathcal B}_{H\bullet}^F)[\Gamma/H])=
{1\over|F|}\sum_{H<F\atop1<C_G(H)}e(\widetilde{\mathcal B}_{H\bullet}^F)=
%=-\sum_{H<F\atop1<C_G(H)}e(\Sigma\widetilde{\mathcal B}_{H\bullet}^F)=
-{1\over|F|}e(\mathcal B_\bullet^F)=-{1\over|F|}
\end{aligned}$$
where $\mathcal B^F$ is the (contractible) poset of those $\{1\leq H<F\mid 1<C_G(H)\}$
so we get
$$f(\chi^\Gamma(\underline{\operatorname{E}}\Gamma))=\sum_{F\in\Omega/\Gamma}{1\over|F|}-{1\over|F|}=0.$$

The formulae of Theorem \ref{mainformula} and Corollary \ref{mainformula2} can be used to compute several Euler characteristics. For example, to compute the (non-equivariant) Euler characteristic of the quotient complex we have to map $[\Gamma/H]\to 1$
$$e(\underline{\operatorname{E}}\Gamma/\Gamma)=\sum_{F\in\Omega/\Gamma}(1+{1\over{|F|}}\sum_{H<F\atop 1<C_G(H)}e(\widetilde{\mathcal B}_{H\bullet}^F)|H|)$$

Or to compute the Euler characteristic defined in string theory (\cite{Adem}, \cite{Thevenaz}, \cite{AtiyahSegal}), where $[\Gamma/H]\mapsto\dim_\C R(H)\otimes\C $ (here, $R(H)$ is the complex representation ring of the finite group $H$):

$$e^{\text{string}}(\underline{\operatorname{E}}\Gamma)=\sum_{F\in\Omega/\Gamma}(\dim_\C R(F)\otimes\C +{1\over{|F|}}\sum_{H<F\atop 1<C_G(H)}e(\widetilde{\mathcal B}_{H\bullet}^F)|H|\dim_\C R(H)\otimes\C ).$$

Assume that $K\leq\GL_n(\Z)$ is a finite group. Then $K$ acts on the $n$-torus $X=\T^n$ and $X/K$ is a toroidal orbifold (see \cite{Adem}, Section 1.2 and Example 2.23). Obviously, $K$ also acts on $A=\Z^n$ and the split extension $\Gamma= K\ltimes A$ associated to this action is precisely the fundamental group of the orbifold $X/K$. It is also the group of lifts to $\R^n$ of the action of $K$ on $\T^n$. As $\Gamma$ acts on $\R^n$ by affine transformations one easily deduces that $\R^n$ is a model for $\egamma$ and $\R^n/\Gamma=\T^n/K$. So the formula of Corollary \ref{mainformula2} can be used to compute $\chi^\Gamma(\underline{\operatorname{E}}\Gamma)$ and $e^{\text{string}}(\R^n/\Gamma)=e^{\text{string}}(\T^n/K)$. We will do this in some examples, in which we use the following notation: Let $K,A,\Gamma$ be as before
and put $\bar\Omega:=\{F\leq K\mid 1<C_A(F)\}$. For each $F\in\bar\Omega$, let $a_F$ be the number of $\Gamma$-conjugacy classes of complements of $A$ in $FA$ which are maximal finite in $\Gamma$. Corollary \ref{mainformula2} yields
$$\chi^\Gamma(\underline{E}\Gamma)=\sum_{F\in\bar\Omega}a_F([K/F]+\sum_{H<F\text{ up to $F$-conjugation}\atop 1<C_G(H)}{1\over|N_F(H):H|}e(\widetilde{\mathcal B}_{H\bullet}^F)[K/H])$$

Moreover, in this case there is an easy way to decide wether for a subgroup $H\leq K$ we have $C_A(H)=1$. As $K$ acts on $A$, by extending scalars we get a $\C K$-module structure in $A\otimes_\Z\C$. Let $\phi$ be the complex character of $K$ associated to this module. Then
\begin{equation}\label{frobenius} C_A(H)=1\text{ if and only if }(\phi|_H,1_H)_H\neq 0
\end{equation}
where $(-,-)_H$ is the ordinary Frobenius product, $\phi|_H$ is the restriction of $\phi$ to $H$ and $1_H$ is the trivial character of $H$.

 For example, we consider the following three orbifolds of \cite{Adem} Examples 1.9,1.10 and 1.11 (also considered in \cite{AdemKlaus} Section 3).

\iffalse{{\bf I should add a brief explanation of why toroidal orbifolds $X\to X/K$ where $K$ is a finite subgroup of $\GL_n(\Z)$ and $X=\R^n/\Z^n$ is the $n$-torus are always $\underline{\operatorname{E}}\Gamma/\Gamma$. Basically, the idea is that $\Gamma=K\ltimes A$ is the group of lifts of elements of $K$ to the covering $\R^n\to X=\T^n=\R^n/A$ for $A=\Z^n$. {\iffalse See for example [Leary] On finite subgroups of groups of type $VF$, pg. 5}\fi Moreover, I think that the argument of [Mislin, Classifying spaces for proper actions of mapping class groups, M\"unster J. of Math. 3 Lemma 4.3] implies that then $\R^n$ is an $\underline{\operatorname{E}}\Gamma$. This could be summarized as follows: any global quotient orbifold which corresponds to a finite of type $\underline{\operatorname{E}}\Gamma$. According to \cite{AdemKlaus} Definition 2.5, a global quotient orbifold is an orbifold given as the quotient of a smooth, effective (i.e.,  $gx=x$ for any$x$ implies $g=1$), almost free (i.e., with finite stabilizers), action of a finite group on a smooth manifold. There is also a book by Adem where maybe I can find a proof, at least the following examples are there: \cite{Adem} Yes, I think this is essentially on Remark 3.19}}\fi

\begin{example} Let $\T^4/K$ be the Kummer surface given by the action of $K=C_2=<x>$ on $\T^4$ via $x=-I_4$. Let $A=\Z^4$ seen as $K$-module with the previous action and $\Gamma=K\ltimes A$. Then $\T^4=\underline{\operatorname{E}}\Gamma/A$. One easy computation yields $|\cohom1KA|=16$, $\bar\Omega=\{K\}$ and $\{H\leq K\mid 1<C_A(H)\}=\{1\}$. So
$$\chi^\Gamma(\R^4)=16([K/K]-{1\over 2}[K/1])=16[K/K]-8[K/1]$$
and
$$e^{\text{string}}(\T^4/K)=32-8=24$$
\end{example}

\begin{example} Let $\T^6/K$ with $K=C_4=<x>$ acting via
$$x=\begin{pmatrix}-I_2&0&0\\
0&U&0\\
0&0&U\\
\end{pmatrix}$$
with $U=\begin{pmatrix}0&-1\\
1&0\\\end{pmatrix}$, $A=\Z^6$ seen as $K$-module and $\Gamma=K\ltimes A$. Then $\T^6=\underline{\operatorname{E}}\Gamma/A$. Now $|\cohom1KA|=16$, $\bar\Omega=\{K\}$ and $\{H\leq K\mid 1<C_A(H)\}=\{1,<x^2>\}$. We get
$$\chi^\Gamma(\R^6)=16([K/K]-{1\over 2}[K/<x^2>])=16[K/K]-8[K/<x^2>],$$
$$e^{\text{string}}(\T^6/K)=16\cdot 4-16=48.$$
\end{example}

\begin{example} Let $\T^6/K$ with $K=C_2\times C_2=<x,y>$ acting via
$$x=\begin{pmatrix}-I_2&0&0\\
0&-I_2&0\\
0&0&I_2\\
\end{pmatrix}, y=\begin{pmatrix}-I_2&0&0\\
0&I_2&0\\
0&0&-I_2\\
\end{pmatrix},$$
 $A=\Z^6$ seen as $K$-module and $\Gamma=K\ltimes A$. Then $\T^6=\underline{\operatorname{E}}\Gamma/A$. Now $|\cohom1KA|=2^6$, $\bar\Omega=\{K\}$ and $\{H\leq K\mid 1<C_A(H)\}=\{1,<x>,<y>,<xy>\}$. With the formula
 $$\chi^\Gamma(\R^6)=2^6([K/K]+{2\over 4}[K/1]-{1\over 2}[K/<x>]-{1/2}[K/<y>]-{1\over 2}[K/<xy>])$$
 $$=64[K/K]+{32}[K/1]-{32}[K/<x>]-{32}[K/<y>]-{32}[K/<xy>],$$
 $$e^{\text{string}}(\T^6/K)=96.$$
\end{example}

\begin{example} Let $\Gamma$ be the group of Example \ref{explicitBredon}.
Using the character $\phi$ associated to the given matrix representation of $F$, an easy computation implies that
$$\{1\leq H<F\mid 1<C_A(H)\}=\{1,<x>,<x^y>,<x^{y^2}>\},$$
$$\{1\leq H<<ya>\mid 1<C_A(H)\}=\{1\}.$$
So the formula of Corollary \ref{mainformula2} yields
$$\begin{aligned}
\chi^\Gamma(\underline{\operatorname{E}}\Gamma)={2\over 6}[\Gamma/1]-[\Gamma/<x>]+[\Gamma/F]-{1\over 3}[\Gamma/1]+[\Gamma/<ya>]=\\
-[\Gamma/<x>]+[\Gamma/F]+[\Gamma/<ya>].
\end{aligned}$$
And from this, we have $e^{\text{string}}(\underline{\operatorname{E}}\Gamma/\Gamma)=4$.

\iffalse{
Let $M=\underline{\operatorname{E}}\Gamma/A$ so that $M$ is a finite $F$-set. In fact from the Euler class of $\underline{\operatorname{E}}\Gamma$ we deduce that the Euler class of $M$ in the Burnside set $A(F)$ is
$$-[F/<x>]+[F/F]+[F/<y>].$$
Now, we compute the Euler classes of the $F$-space $M$, the $C_F(x)=<x>$-space $M^{<x>}$ and the $C_F(y)=<y>$-space $M^{<y>}$ and get:
\begin{enumerate}
\item For the $F$-space $M$, $-[F/<x>]+[F/F]+[F/<y>]$ thus
$e(M/F)=-1+1+1=1$.

\item For the $<x>$-space $M^{<x>}$, $-[F/<x>]+[F/F]=-[<x>/<x>]-[<x>/1]+[<x>/<x>]\in A(<x>)$ thus $e(M^{<x>}/<x>)=-1-1+1=-1$.

\item For the $<y>$-space $M^{<y>}$, $[F/F]+[F/<y>]=[<y>/<y>]+3[<y>/<y>]\in A(<y>)$ thus $e(M^{<y>}/<y>)=1+3=4$.
\end{enumerate}
So we get
$$\sum_{[g]\in F}e(M^{<g>}/C_F(g))=1-1+4=4.$$}\fi

\end{example}

\begin{example} Let $K=A_5$, the alternating group on 5 letters act on $A=\Z^4$ via
$$x=(12)(34)\mapsto\begin{pmatrix}
1&0&0&0\\
0&0&1&0\\
0&1&0&0\\
-1&-1&-1&-1\\
\end{pmatrix},$$
$$y=(135)\mapsto\begin{pmatrix}
0&1&0&0\\
0&0&0&1\\
0&0&1&0\\
1&0&0&0\\
\end{pmatrix}$$
and again consider the associated action of $K$ on $\T^4$. Now, $\cohom1KA=0$ so $a_K=1$ and $\{H\leq K\mid C_A(H)=1\}$ consists of all the subgroups of  $K$ except of all the 5-Sylows and their normalizers (this is easily computed using the character table and (\ref{frobenius})). Moreover, if $R$ is one of the 5-Sylow subgroups, then $\cohom1RA=\Z/5\Z$ and $\cohom1{N_K(R)}A=0$. From this we see that $\bar\Omega=\{K,R\}$. Under the action of $N_K(R)$ there are exactly two non-trivial orbits in $\cohom1RA$ thus $a_R=2$. Denote $C=<x>$, $P$ for a 2-Sylow and $Q$ for a 3-Sylow. Observe that $N_K(C)\cong P$, $N_K(P)\cong A_4$, $N_K(Q)\cong S_3$, $N_K(N_K(P))=N_K(P)$, $N_K(N_K(Q))=N_K(Q)$ so

$$\begin{aligned}
%[K/K]+{1\over|K|}e(\widetilde{\mathcal B}_1^K)[K/1]+{1\over|N_K(C):C|}e(\widetilde{\mathcal B}_{<x>}^K)[K/<x>]+\\
%{1\over|N_K(P):P|}e(\widetilde{\mathcal B}_P^K)[K/P]+e(\widetilde{\mathcal B}_{N_K(P)}^K)[K/N_K(P)]+\\
%{1\over|N_K(Q):Q|}e(\widetilde{\mathcal B}_Q^K)[K/Q]+e(\widetilde{\mathcal B}_{N_K(Q)}^K)[K/N_K(Q)]+\\
%[K/R]+{1\over|R|}e(\widetilde{\mathcal B}_1^R)[K/1]=\\
\chi^K(\T^4)=[K/K]+{1\over 60}e(\widetilde{\mathcal B}_{1\bullet}^K)[K/1]+{1\over2}e(\widetilde{\mathcal B}_{C\bullet}^K)[K/C]+
{1\over 3}e(\widetilde{\mathcal B}_{P\bullet}^K)[K/P]+\\
e(\widetilde{\mathcal B}_{N_K(P)\bullet}^K)[K/N_K(P)]+
{1\over 2}e(\widetilde{\mathcal B}_{Q\bullet}^K)[K/Q]+e(\widetilde{\mathcal B}_{N_K(Q)\bullet}^K)[K/N_K(Q)]+\\
2([K/R]+{1\over 5}e(\widetilde{\mathcal B}_{1\bullet}^R)[K/1]).\\
\end{aligned}$$
Also, $e(\widetilde{\mathcal B}_{1\bullet}^K)=-36$, $e(\widetilde{\mathcal B}_{C\bullet}^K)=2$, $e(\widetilde{\mathcal B}_{P\bullet}^K)=0$, $e(\widetilde{\mathcal B}_{N_K(P)\bullet}^K)=-1$, $e(\widetilde{\mathcal B}_{Q\bullet}^K)=2$, $e(\widetilde{\mathcal B}_{N_K(Q)\bullet}^K)=-1$ $e(\widetilde{\mathcal B}_{1\bullet}^R)=-1$
so the formula gives us:
$$\begin{aligned}
\chi^K(\T^4)=[K/K]-[K/1]+[K/C]
-[K/N_K(P)]+[K/Q]-[K/N_K(Q)]+
2[K/R].\\
\end{aligned}$$
and
$e^{\text{string}}(\T^4/K)=5-1+2-4+3-3+10=12$.
%Comprobaci—n: $1/60-1+1/2-1/12+1/3-1/6+2/5=75/60-5/4=0$.
\end{example}

%\bibliography{posetfinsbgps}

\begin{thebibliography}{1000}

\bibitem{AdemKlaus} A.~Adem, M.~Klaus, {\em Lectures on orbifolds and group cohomology, Advanced Lectures in Mathematics 16, Transformation Groups and Moduli Spaces of Curves} Higher Education Press, edited by L.Ji and S-T Yau, 2010.

\bibitem{Adem} A.~Adem, J.~Leida, Y.~Ruan, Orbifolds and string topology. Cambridge Tracts in Mathematics, Vol. 171, Cambridge Un. Press, 2007.

\bibitem{AtiyahSegal} M.~F.~Atiyah, G.~Segal, On equivariant Euler characteristics. J. Geom. Physics 6 (4), (1989), 671-677.


\bibitem{BenI} D.~J.~Benson.
\newblock {\em Representations and cohomology I:
 Basic representation theory of finite groups
and associative algebras.}
 \newblock Cambridge University Press, Cambridge 1991.



\bibitem{Brown} K.~S.~Brown.
\newblock {\em Cohomology of groups.}
\newblock Springer-Verlag, New York 1982.


\bibitem{Brown2} K.~S.~Brown. Complete Euler characteristics and fixed-point theory. J. of Pure and Appl. Algebra, 24 (1982), 103-121.

\bibitem{connkoz}
F.~Connolly and T.~Kozniewski.
\newblock Finiteness properties of classifying spaces of proper
  $\gamma$-actions.
\newblock  J. Pure Appl. Algebra, 41 (1), (1986), 17-36.

\bibitem{corkro}
J.~Cornick and P.~H.~Kropholler.
\newblock Homological finiteness conditions for modules over group algebras.
\newblock J. London Math. Soc., 58 (2), (1998), 49-62.


\bibitem{Eckmann} B.~Eckmann. Amenable groups and Euler characteristic, Comment. Math. Helv. 67
(1992), 383-393.



\bibitem{kmn}
P.~Kropholler, C.~Mart\'inez-P\'erez, and B.~E.~A. Nucinkis, Cohomological finiteness conditions for elementary amenable groups, J. Reine Angew. Math. 637, (2009), 49-62.

\bibitem{krophollermislin}
P.~H.~Kropholler and G.~Mislin.
\newblock Groups acting on finite-dimensional spaces with finite stabilizers.
\newblock  Comment. Math. Helv., 73 (1), (1998), 122-136.


\bibitem{krophollertalelli} P.~Kropholler and O.~Talelli.
\newblock On a property of fundamental groups of graphs of finite groups.
\newblock  J. Pure Appl. Algebra 74, (1991), 57-59.

\bibitem{LangerLuck} M.~Langer, W.~L\"uck, On the group cohomology of the semi-direct product $\Z^n\rtimes_\rho\Z_m$ and a conjecture of Adem-Ge-Petrosyan.
http://arxiv.org/pdf/1105.4772v1


\bibitem{LN}  I.J. Leary and B.E.A. Nucinkis. Some groups of type VF. Invent. Math. 151 (1) (2003), 135-165.





\bibitem{LennoxRobinson}
J.~C.~Lennox and D.~J.~S.~Robinson.
\newblock {\em The Theory of Infinite Soluble Groups.}
\newblock Oxford Un. Press, New York, 2004.

\bibitem{lueckbook}
W.~L{\"u}ck.
\newblock {\em Transformation groups and algebraic {$K$}-theory}, volume 1408
  of {\em Lecture Notes in Mathematics}.
\newblock Springer-Verlag, Berlin, 1989.


\bibitem{lueck}
W.~L{\"u}ck.
\newblock The type of the classifying space for a family of subgroups.
\newblock  J. Pure Appl. Algebra, 149 (2), (2000), 177-203.

\bibitem{luckL2} W.~L{\"u}ck.
\newblock {\em $L^2$-Invariants: Theory and Applications to Geometry and K-Theory. Ergebnisse der Mathematik und ihrer Grenzgebiete} 44, Springer, 2002.


\bibitem{conch}
C.~Mart\'inez-P\'erez.
\newblock A bound for the Bredon cohomological dimension.
\newblock J. Group Theory, 10 (6), (2007), 731-747.

\bibitem{martineznucinkis} C. Mart\'inez-P\'erez and B. E. A. Nucinkis. Virtually soluble groups of type $\fpinfty$,
\newblock Comment. Math. Helv., 85 (1), (2010), 135-150.



\bibitem{symonds}
P.~Symonds.
\newblock The Bredon cohomology of subgroup complexes.
\newblock J. Pure Appl. Algebra, 199(1-3), (2005), 261-298.

\bibitem{Thevenaz} J.~Thevenaz. Equivariant $K$-theory and Alperin's conjecture, J. of Pure Appl. Algebra, 85 (2), (1993), 185-202.

\bibitem{Weibel} Ch.~A.~Weibel.
\newblock {\em An introduction to homological algebra},
\newblock Cambridge University Press, Cambridge 1994.
\end{thebibliography}
\bibliographystyle{abbrv}

\end{document}